\documentclass[oneside,english]{amsart}
\usepackage[T1]{fontenc}
\usepackage[latin9]{inputenc}
\usepackage{float}
\usepackage{amstext}
\usepackage{amsthm}
\usepackage{amssymb}
\usepackage[all]{xy}

\makeatletter

\providecommand{\tabularnewline}{\\}

\numberwithin{equation}{section}
\numberwithin{figure}{section}
\theoremstyle{plain}
\newtheorem{thm}{\protect\theoremname}
\theoremstyle{definition}
\newtheorem{defn}[thm]{\protect\definitionname}
\theoremstyle{plain}
\newtheorem{lem}[thm]{\protect\lemmaname}
\theoremstyle{plain}
\newtheorem{prop}[thm]{\protect\propositionname}
\theoremstyle{definition}
\newtheorem{problem}[thm]{\protect\problemname}
\theoremstyle{definition}
\newtheorem{example}[thm]{\protect\examplename}
\theoremstyle{plain}
\newtheorem{cor}[thm]{\protect\corollaryname}
\theoremstyle{remark}
\newtheorem{rem}[thm]{\protect\remarkname}

\makeatother

\usepackage{babel}
\providecommand{\corollaryname}{Corollary}
\providecommand{\definitionname}{Definition}
\providecommand{\examplename}{Example}
\providecommand{\lemmaname}{Lemma}
\providecommand{\problemname}{Problem}
\providecommand{\propositionname}{Proposition}
\providecommand{\remarkname}{Remark}
\providecommand{\theoremname}{Theorem}

\begin{document}
\global\long\def\G{\mathcal{G}}%
 
\global\long\def\F{\mathcal{F}}%
 
\global\long\def\H{\mathcal{H}}%
\global\long\def\Z{\mathcal{Z}}%
 
\global\long\def\L{\mathcal{L}}%
\global\long\def\U{\mathcal{U}}%
\global\long\def\W{\mathcal{W}}%
 
\global\long\def\E{\mathcal{E}}%
\global\long\def\B{\mathcal{B}}%
 
\global\long\def\A{\mathcal{A}}%
\global\long\def\D{\mathcal{D}}%
\global\long\def\O{\mathcal{O}}%
 
\global\long\def\N{\mathcal{N}}%
 
\global\long\def\X{\mathcal{X}}%
 
\global\long\def\lm{\lim\nolimits}%
 
\global\long\def\then{\Longrightarrow}%

\global\long\def\V{\mathcal{V}}%
\global\long\def\C{\mathcal{C}}%
\global\long\def\adh{\operatorname{adh}\nolimits}%
\global\long\def\Seq{\operatorname{Seq}\nolimits}%
\global\long\def\intr{\operatorname{int}\nolimits}%
\global\long\def\cl{\operatorname{cl}\nolimits}%
\global\long\def\inh{\operatorname{inh}\nolimits}%
\global\long\def\diam{\operatorname{diam}\nolimits\ }%
\global\long\def\card{\operatorname{card}}%
\global\long\def\T{\operatorname{T}}%
\global\long\def\S{\operatorname{S}}%
\global\long\def\id{\operatorname{id}}%

\global\long\def\fix{\operatorname{fix}\nolimits}%
\global\long\def\Epi{\operatorname{Epi}\nolimits}%

\global\long\def\Conv{\mathbf{Conv}}%
\global\long\def\Cap{\mathbf{Cap}}%
\global\long\def\VCap{V\text{-}\mathbf{Cap}}%

\global\long\def\VCcap{V^{(-)}\text{-}\Cap}%

\global\long\def\fix{\operatorname{fix}\nolimits}%
\global\long\def\Epi{\operatorname{Epi}\nolimits}%

\global\long\def\Cconv{\mathbf{C}^{\mathrm{conv}}}%
\global\long\def\op#1{\operatorname#1}%
\global\long\def\pt{\operatorname{pt}\nolimits}%
\global\long\def\ptS{\operatorname{pt}\nolimits _{\mathrm{Stone}}}%
\global\long\def\ppt{\operatorname{ppt}}%

\title[points of convergence lattices]{on points of convergence lattices and sobriety for convergence spaces }
\author{Frédéric Mynard}
\address{NJCU, Department of Mathematics, 2039 Kennedy Blvd Jersey City, NJ
07305, USA}
\email{fmynard@njcu.edu}
\begin{abstract}
We characterize the convergence spaces $(X,\xi)$ such that the space
of points of $(\mathbb{P}X,\lim_{\xi})$ in the category of convergence
lattices is $(X,\xi)$. On the way, we study variants of sobriety
and of the axiom $T_{D}$ in convergence spaces. New phenomena appear
when leaving the realm of topological spaces. We obtain new hindsight
into the space of points of a convergence lattice and study a special
quotient of it, which, in the case $L=(\mathbb{P}X,\lim_{\xi})$ for
a topological space $(X,\xi)$, turns out to be homeomorphic to the
sobrification of $X$.
\end{abstract}

\maketitle

\section{Preliminaries and introduction}

\subsection{Convergence Spaces}

Let $\mathbb{P}X$ denote the powerset of $X$. If $\A\subset\mathbb{P}X$,
\[
\A^{\uparrow}=\left\{ B\subset X:\exists A\in\A,A\subset B\right\} 
\]
and 
\[
\A^{\#}=\left\{ B\subset X:\forall A\in\A,A\cap B\neq\emptyset\right\} .
\]
 We also write $\A\#\B$, and say that $\A$ and $\B$ mesh, if $\B\subset\A^{\#}$,
equivalently, $\A\subset\B^{\#}$.

A convergence $\xi$ on a set $X$ is a relation between the set $\mathbb{FP}X$
of (set-theoretic) filters on $X$ and the set $X$, denoted $x\in\lm_{\xi}\F$
(or $x\in\lim\F$ if there's no risk of ambiguity) if $\F$ and $x$
are $\xi$-related, subject to the following two axioms:

\begin{equation}
\tag{monotone}\text{\ensuremath{\F}\ensuremath{\ensuremath{\subset\G\then\lm_{\xi}\F\subset\lm_{\xi}\G}}}\label{eq:ConvMonotone}
\end{equation}
\begin{equation}
\tag{point axiom}x\in\lm_{\xi}\{x\}^{\uparrow}\label{eq:ptaxiom}
\end{equation}
 for every $x\in X$ and every $\F,\G\in\mathbb{FP}X$. A convergence
is \emph{of finite depth }if additionally
\begin{equation}
\tag{finite depth}\lm_{\xi}(\F\cap\G)=\lm_{\xi}\F\cap\lm_{\xi}\G\label{eq:finitedepth}
\end{equation}
for every $\F,\G\in\mathbb{FP}X$.

Continuity of a map $f:(X,\xi)\to(Y,\tau)$, in symbols $f\in C(\xi,\tau)$,
is simply preservation of limits, that is,
\begin{equation}
f(\lm_{\xi}\F)\subset\lm_{\tau}f[\F],\tag{ continuity}\label{eq:continuity}
\end{equation}
where $f[\F]=\{B\subset Y:f^{-}(B)\in\F\}\in\mathbb{FP}Y$ is the
image filter of $\F$ under $f$.

Let $\Conv$ denote the category of convergence spaces and continuous
maps. We denote by $|\cdot|:\mathbf{Conv}\to\mathbf{Set}$ the forgetful
functor, so that $|\xi|$ denotes the underlying set of a convergence
$\xi$ and $|f|$ is the underlying function of a morphism. If $|\xi|=|\tau|$,
we say that $\xi$ is \emph{finer than $\tau$ }or that $\tau$ is
\emph{coarser than }$\xi$, in symbols, $\xi\geq\tau$, if the identity
map of their underlying set belongs to $C(\xi,\tau)$. This order
turns the set of convergences on a given set into a complete lattice
whose greatest element is the discrete topology, least element is
the antidiscrete topology, and whose suprema and infima are given
by
\begin{equation}
\lm_{\bigvee_{\xi\in\Xi}\xi}\F=\bigcap_{\xi\in\Xi}\lm_{\xi}\F\text{ and }\lm_{\bigwedge_{\xi\in\Xi}\xi}\F=\bigcup_{\xi\in\Xi}\lm_{\xi}\F.\label{eq:latticeconv-1}
\end{equation}

$\mathbf{Conv}$ is a concrete topological category; in particular,
for every $f:X\to|\tau|$, there is the coarsest convergence on $X$,
called \emph{initial convergence }for\emph{ $(f,\tau)$ }and denoted
$f^{-}\tau$, making $f$ continuous (to $\tau$), and for every $f:|\xi|\to Y$,
there is the finest convergence on $Y$, called \emph{final convergence
for }$(f,\xi)$ and denoted $f\xi$, making $f$ continuous (from
$\xi$). Note that with these notations
\begin{equation}
f\in C(\xi,\tau)\iff\xi\geq f^{-}\tau\iff f\xi\geq\tau.\label{eq:continuity-1-1}
\end{equation}

Moreover, the initial lift on $X$ of a structured source $(f_{i}:X\to|\tau_{i}|)_{i\in I}$
turns out to be $\bigvee_{i\in I}f_{i}^{-}\tau_{i}$ and the final
lift on $Y$ of a structured sink $(f_{i}:|\xi_{i}|\to Y)_{i\in I}$
turns out to be $\bigwedge_{i\in I}f_{i}\xi_{i}$.

If $f:|\xi|\to Y$ is surjective, $f\xi$ is also called the \emph{quotient
convergence. }If $A\subset|\xi|$, the \emph{subspace convergence
}$\xi_{|A}$ or \emph{induced convergence on $A$ }is $i_{A}^{-}\xi$
where $i_{A}:A\to|\xi|$ is the inclusion map. If $\Xi$ is a set
of convergences, then the \emph{product convergence} is
\[
\Pi_{\xi\in\Xi}\xi=\bigvee_{\xi\in\Xi}p_{\xi}^{-}\xi,
\]
where $p_{\xi}:\Pi_{\xi\in\Xi}|\xi|\to|\xi|$ is the projection, that
is, $\Pi_{\xi\in\Xi}\xi$ is the initial convergence for the family
of projections.

Convergences of finite depth form a concretely reflective subcategory
of $\Conv$ with reflector $\operatorname{L}$ defined (on objects)
by $x\in\lim_{\operatorname{L}\xi}\F$ if there is a finite set $\mathbb{D}$
filters on $X$ each converging to $x$ for $\xi$ such that $\F\geq\bigcap_{\D\in\mathbb{D}}\D$.

Of course, every topology $\tau$ can be seen as a convergence given
by $x\in\lim_{\tau}\F$ if and only if $\F\geq\N_{\tau}(x)$, where
$\N_{\tau}(x)$ denotes the neighborhood filter of $x$ in the topology
$\tau$. This turns the category $\mathbf{Top}$ of topological spaces
and continuous maps into a full (concretely reflective) subcategory
of (the topological extensional and cartesian-closed category) $\mathbf{Conv}$.
If $(X,\xi)$ is a convergence space, the \emph{topological modification
$\T\xi$ }of $\xi$ is the topology on $X$ whose set of closed sets
are the $\xi$-\emph{closed} subsets of $X$, that is, the sets $C\subset X$
with 
\begin{equation}
C\in\F\then\lm_{\xi}\F\subset C.\label{eq:closed}
\end{equation}

The $\xi$-\emph{open }subsets are complements of $\xi$-closed sets,
that is, subsets $O$ satisfying
\begin{equation}
\lm_{\xi}\F\cap O\neq\emptyset\then O\in\F.\label{eq:open}
\end{equation}

$\T\xi$ is the finest among topologies coarser than $\xi$. Note
that if $\O_{\xi}$ denotes the set of $\xi$-open sets and $\O_{\xi}(x)=\O_{\xi}\cap\{x\}^{\uparrow}$,
then $\N_{\xi}(x)=(\O_{\xi}(x))^{\uparrow}$.

A subspace $(A,\xi_{|A})$ of a convergence space $(X,\xi)$ is \emph{dense
}if for every $x\in X$ there is $\F\in\mathbb{F}X$ with $A\in\F$
and $x\in\lim_{\xi}\F$.

\subsection{Convergence Lattices}

A pointfree generalization was offered in \cite{FredetJean} in which
the function 
\[
\lm_{\xi}:\mathbb{FP}X\to\mathbb{P}X
\]
is abstracted away to a monotone function 
\[
\lim:\mathbb{F}L\to L
\]
from (order-theoretic) filters on a lattice $L$ to $L$. (\ref{eq:ptaxiom})
is not part of the axiomatic in this pointfree version of convergence
spaces, though the notion can also be recovered (as so-called \emph{centered
}convergence lattices) in an abstract order-theoretic form. 
\begin{defn}
Given a category $\mathbf{C}$ of lattices, a\emph{ convergence $\mathbf{C}$-object
$(L,\lim)$ }is a $\mathbf{C}$-object $L$ together with a monotone
map $\lim:\mathbb{F}L\to L$. The objects of the category $\Cconv$
are the convergence $\mathbf{C}$-objects and the morphisms $\varphi:L\to L'$
are the $\mathbf{C}$-morphisms that are \emph{continuous }in the
sense that for every $\F\in\mathbb{F}L'$,
\begin{equation}
\lm_{L'}\F\leq\varphi(\lm_{L}\varphi^{-1}(\F)),\tag{ptfree continuity}\label{eq:ptfreeCont}
\end{equation}
 where $\varphi^{-1}(\F)=\{\ell\in L:\varphi(\ell)\in\F\}$.
\end{defn}

The category $\Conv$ embeds coreflectively into $(\Cconv)^{op}$
when $\mathbf{C}$ is the category of frames or of coframes: the \emph{powerset
functor} $\mathbb{P}:\Conv\to(\Cconv)^{op}$ sending $(X,\xi)$ to
$(\mathbb{P}X,\lm_{\xi})$ and $f:(X,\xi)\to(Y,\tau)$ to $\text{\ensuremath{\mathbb{P}f=}}f^{-1}:\mathbb{P}Y\to\mathbb{P}X$
(in $\Cconv$) is then right-adjoint to the \emph{point-functor} $\pt:(\Cconv)^{op}\to\Conv$
(the coreflector) where the underlying set of $\pt L$, the set of
``points'' of $L$, is the set of $\Cconv$-morphisms from $L$
to $\mathbb{P}(1)$ (which is terminal in $(\Cconv)^{op}$ \cite[Lemma 2.7]{FredetJean}),
hence depends on the choice of $\mathbf{C}$, so we'll write $\pt_{\mathbf{C}}L$
when different options for the base category $\mathbf{C}$ of lattices
are considered.

For a given $\mathbf{C}$, the convergence structure on $\pt L$ is
given by 
\begin{equation}
\tag{Conv on pt}\lm_{\pt L}\F=\left(\lm_{L}\F^{\circ}\right)^{\bullet},\label{eq:ptconv}
\end{equation}
where $\ell^{\bullet}=\{\varphi\in\pt L:\varphi(\ell)=\{1\}\}$ and
$\F^{\circ}=\left\{ \ell\in L:\ell^{\bullet}\in\F\right\} $. Finally,
if $\varphi\in\Cconv(L,L')$ then $\pt\varphi:\pt L'\to\pt L$ is
defined by $\pt(\varphi)(f)=f\circ\varphi$.

As pointed out in \cite[Remark 2.9]{FredetJean}, a continuous $\mathbf{C}$-morphism
$\varphi:L\to\mathbb{P}(1)$ can be identified with the filter $\F=\varphi^{-1}(\{1\})$
which satisfies $\lim_{L}\F\in\F$ by continuity. Hence the set $\pt_{\mathbf{C}}L$
of points of a $\Cconv$-object $L$ can be identified with sets of
specific filters as follows (\footnote{Recall that a proper filter $\F$ on a lattice $L$ is \emph{prime
}if $\ell_{1}\vee\ell_{2}\in\F$ implies that $\ell_{1}$ or $\ell_{2}$
belongs to $\F$ and \emph{completely prime }if for every $A\subset L$,
$\bigvee_{\ell\in A}\ell\in\F$ implies that $\ell\in\F$ for some
$\ell\in A$. An element $\ell\in L$ is \emph{join-prime }if $\ell\leq\ell_{1}\vee\ell_{2}$
implies $\ell\leq\ell_{1}$ or $\ell\leq\ell_{2}$.}):\medskip{}

\begin{center}
\begin{tabular}{|c|c|c|}
\hline 
cat. $\mathbf{C}$ of & $\pt_{\mathbf{C}}L=\{\F\in\mathbb{F}L:\lim_{L}\F\in\F\text{ and...}\}$ & ref.\tabularnewline
\hline 
\hline 
lattices & $\F$ is prime & \cite[Remark 2.10]{FredetJean}\tabularnewline
\hline 
frames & $\F$ is completely prime & \cite[Remark 2.11]{FredetJean}\tabularnewline
\hline 
coframes & $\F$ is principal of a $\ensuremath{\vee}$-prime element & \cite[Remark 2.12]{FredetJean}\tabularnewline
\hline 
\end{tabular}
\par\end{center}

\medskip{}

\begin{lem}
\label{lem:principalofpoints}\cite[Lemma 2.19(6)]{FredetJean} Let
$L$ be a convergence lattice and let $\U\in\pt L$. Then $(\{\U\}^{\uparrow})^{\circ}=\U$.
\end{lem}

\begin{proof}
$\ell\in(\{\U\}^{\uparrow})^{\circ}$ if and only if $\ell^{\bullet}\in\{\U\}^{\uparrow}$,
that is, if and only if $\U\in\ell^{\bullet}$, equivalently, $\ell\in\U$.
\end{proof}
It turns out that filters of the form $\{x\}^{\uparrow}$ are always
points of $\mathbb{P}(X)$ in $\Cconv$ provided that $\mathbf{C}$
be an admissible category (\footnote{A category $\mathbf{C}$ of lattices is \emph{admissible }if $\mathbb{P}(X)$
is a $\mathbf{C}$-object for every set $X$ and there are two classes
$\mathcal{I}$ and $\mathcal{J}$ of index sets such that for all
pairs $L,L'$ of $\mathbf{C}$-objects, the $\mathbf{C}$-morphisms
from $L$ to $L'$ are exactly monotonic maps that preserve all $I$-indexed
infima that exist in $L$ for $I\in\mathcal{I}$ andall mJ -indexed
suprema that exist in $L$ for $J\in\mathcal{J}$. See \cite[Def. 2.5]{FredetJean} }) of lattices \cite[Remark 2.14]{FredetJean}. As a result (\footnote{Save density, this was already observed in \cite[Lemma 2.26]{FredetJean}}):
\begin{prop}
\label{prop:subspaceofptL} A convergence space $(X,\xi)$ is always
(homeomorphic to) a dense  subspace of $\pt((\mathbb{P}X,\lim_{\xi}))$.
\end{prop}

\begin{proof}
The map $h:(X,\xi)\to\pt((\mathbb{P}X,\lim_{\xi}))$ defined by $h(x)=\{x\}^{\uparrow}$
is an embedding, i.e., an homeomorphism onto its image. The map $h$
is clearly one-to-one. To see that it is an embedding, let $L=\pt((\mathbb{P}X,\lim_{\xi}))$
and recall that $\lim_{\pt L}h[\F]=(\lim_{\xi}h[\F]^{\circ})^{\bullet}$
where $h[\F]^{\circ}=\{A\in\mathbb{P}(X):A^{\bullet}\in h[\F]\}$
and $A^{\bullet}=\{\G\in\pt L:A\in\G\}$. Noting that $A^{\bullet}\in h[\F]$
if and only if $A\in\F$, we see that $h(x)\in\lim_{\pt L}h[\F]$
if and only if $x\in\lim_{\xi}\F$.

Suppose $\G\in\pt L$. Then $\lim_{\xi}\G\in\G$ (when $\G$ is seen
as an element of $\mathbb{FP}X$), that is, $\G\in(\lim_{\xi}\G)^{\bullet}$.
Moreover $\G\subset h[\G]^{\circ}$ so that $\lim_{\xi}\G\subset\lim_{\xi}h[\G]^{\circ}$
and thus $\G\in(\lim_{\xi}\G)^{\bullet}\subset(\lim_{\xi}h[\G]^{\circ})^{\bullet}$,
that is, $\G\in\lim_{\pt L}h[\G]$. Hence $h(X)$ is a dense subspace
of $\pt L$. To see that $\G\subset h[\G]^{\circ}$, note that if
$G\in\G$ then 
\[
h(G)=\{\{x\}^{\uparrow}:x\in G\}\subset G^{\bullet}
\]
so that $G^{\bullet}\in h[\G]$, that is, $G\in h[\G]^{\circ}$.
\end{proof}

As filters of the form $\{x\}^{\uparrow}$ are exactly completely
prime filters in $\mathbb{P}(X)$ they are the only points in $\Cconv$
if $\mathbf{C}$ is the category of frames \cite[Remark 2.16]{FredetJean}.
Similarly, as join-prime elements of $\mathbb{P}(X)$ are singletons,
points of $\mathbb{P}(X)$ in $\Cconv$ are exactly singletons if
$\mathbf{C}$ is the category of coframes \cite[Remark 2.17]{FredetJean}.
Therefore, in both cases, $\pt((\mathbb{P}X,\lim_{\xi}))$ is homeomorphic
to $(X,\xi)$.

On the other hand, prime filters of $\mathbb{P}(X)$ are the \emph{ultrafilters}
on $X$. Hence, when $\mathbf{C}$ is the category $\mathbf{Lat}$
of lattices and $(X,\xi)$ is a convergence space, the set of points
of $(\mathbb{P}(X),\lim_{\xi})$ in $\Cconv$ is
\begin{equation}
\pt_{\mathbf{Lat}}\mathbb{P}(X)=\left\{ \U\in\mathbb{U}X:\lm_{\xi}\U\in\U\right\} ,\label{eq:pointslattices}
\end{equation}
where $\mathbb{U}X$ denotes the set of ultrafilters on $X$ \cite[Remark 2.15]{FredetJean}.
Though it contains the ``usual points'' identified with their principal
ultrafilter $\{x\}^{\uparrow}$, there may in general be others (See
Lemma \ref{lem:Noetherian} below). One purpose of this note is to
clarify when this does not happen, that is:
\begin{problem}
\label{prob:main} Under what condition on a convergence space $(X,\xi)$
are all ultrafilters $\U\in\mathbb{U}X$ with $\lim_{\xi}\U\in\U$
principal? 

First we should point out that the question at hand was fully solved
\cite[Theorem 2.2 with Addendum I]{hoffmann1977irreducible} for the
case where $(X,\xi)$ is topological. Though our motivations leading
to Problem \ref{prob:main} are distinct from those of \cite{hoffmann1977irreducible}
and come from \cite{FredetJean}, we are thus essentially exploring
a generalization of \cite[Theorem 2.2]{hoffmann1977irreducible} from
topological spaces to convergence spaces. Such an extension may seem
of little relevance, but in light of the question of finding out when
$\pt_{\mathbf{Lat}}\mathbb{P}(X)$ are just the ``usual points of
$X$'' in the context of convergence lattices, this is a most natural
problem and restricting ourselves to topological spaces is then highly
undesirable. Corollary \ref{cor:soberTd} gives a full answer to Problem
\ref{prob:main} for a large class of convergence spaces.

It turns out that some interesting subtleties appear in this wider
context of convergence spaces. On the way, we study variants of sobriety
for convergence spaces and take the opportunity to correct a related
misstatement in \cite[Remark 2.15]{FredetJean}.

Though Problem \ref{prob:main} stems from a question on $\pt L$
in the very particular case where $L=(\mathbb{P}X,\lim_{\xi})$ for
a convergence space $(X,\xi)$, we study more in depth the space $\pt L$
for a general convergence lattice $L$. A particular natural quotient
$\pt'L$ of $\pt L$ turns out to be of importance. In particular,
when $L=(\mathbb{P}X,\lim_{\xi})$ for a topological space $(X,\xi)$,
the space $\pt'L$ is homeomorphic to the sobrification of $(X,\xi)$.
\end{problem}

\section{Irreducible filters, sobriety, and convergence variants\label{sec:sober}}

A filter $\F$ on a convergence space $(X,\xi)$ is \emph{irreducible
}in the sense of \cite{hoffmann1977irreducible} if 
\[
\lm_{\xi}\F\in\F.
\]

Hence, elements of $\pt_{\mathbf{C}}\mathbb{P}(X)$ are specific types
of irreducible filters. As we have seen, when $\mathbf{C}$ is either
the category of frames or that of coframes, points turn out to be
exactly principal ultrafilters without the irreducibility condition
even playing a role. This is not the case for $\pt_{\mathbf{Lat}}\mathbb{P}(X)$
whose elements are exactly irreducible ultrafilters. Note that if
$\F$ is irreducible and $\G\geq\F$ then $\G$ is irreducible too
and if $f:(X,\xi)\to(Y,\tau)$ is continuous and $\F$ is irreducible
on $X$ then $f[\F]$ is irreducible on $Y$.

The notion is related to that of a \emph{compact filter }in the sense
of \cite{D.comp,DGL,myn.relations,myn.applofcompact}. A filter $\F$
on a convergence space $(X,\xi)$ is \emph{compact }(resp. \emph{compactoid})
if 
\begin{equation}
\H\#\F\then\adh\H\#\F\text{ }(\text{resp. }\adh\H\neq\emptyset),\label{eq:compactness}
\end{equation}
for every $\H\in\mathbb{F}X$. 
\begin{lem}
Let $(X,\xi)$ be a convergence space.
\begin{enumerate}
\item If $\F$ is irreducible then $\F$ is compact;
\item A compact ultrafilter is irreducible.
\end{enumerate}
\end{lem}

\begin{proof}
(1). If $\F$ is irreducible and $\H\#\F$ then $\lim\F\subset\adh\H$
so that $\adh\H\in\F\subset\F^{\#}$.

(2). If $\U\in\mathbb{U}X$ is compact, then taking $\H=\U$ in (\ref{eq:compactness})
we obtain $\lim\U\in\U^{\#}=\U$ and thus $\U$ is irreducible.
\end{proof}
Hence an ultrafilter is irreducible if and only if it is compact.
However, there are compact filters that are not irreducible. For instance,
a neighborhood filter of the real line with its usual topology is
compact but not irreducible.

Recall that a topological space $X$ is called \emph{Noetherian }if
every decreasing sequence of closed subsets is eventually constant.
It is well-known that $X$ is Noetherian if and only if all of its
subsets are compact. In \cite[Remark 2.15]{FredetJean}, it is erroneously
stated that a topological space is Noetherian if and only if every
compact ultrafilter is principal. On the contrary:
\begin{lem}
\label{lem:Noetherian} A topological space is Noetherian if and only
if every filter, equivalently every ultrafilter, is compact. In particular,
an infinite Noetherian topological space admits irreducible free ultrafilters.
\end{lem}

\begin{proof}
If $X$ is Noetherian, $\F\in\mathbb{F}X$ and $\H$ is a filter $\H\#\F$
then $\adh\H\cap F\neq\emptyset$ for every $F\in\F$ because $F$
is compact. Conversely, if every ultrafilter is compact and $A\subset X$
then $\lim\U\cap A\neq\emptyset$ whenever $A\in\U\in\mathbb{U}X$
because $\adh\U=\lim\U\in\U^{\#}$. 
\end{proof}
Recall that a closed subset $C$ of a topological space is \emph{irreducible
}if whenever $C\subset D\cup F$ where $D$ and $F$ are closed, then
either $C\subset D$ or $C\subset F$. A topological space $X$ is
\emph{sober} if for every irreducible closed subset $C$ of $X$,
there is a (necessarily unique) point $x\in X$ with $C=\cl\{x\}$,
which we call a \emph{generic point for $C$}. Note that we can equivalently
formulate irreducibility of $C$ in the following alternative terms:
If $O_{1},O_{2}$ are open sets intersecting with $C$ then $O_{1}\cap O_{2}$
also intersects with $C$.

It was shown in \cite{hoffmann1977irreducible} that a topological
space $X$ is sober if and only if for every irreducible filter $\F$
on $X$ there is a unique point $x\in X$ with $\lim\F=\lim\{x\}^{\uparrow}$.
Following \cite{hoffmann1977irreducible}, we can take this as a definition
of a \emph{sober convergence space }and call \emph{weakly sober} a
convergence space in which for every irreducible\emph{ ultrafilter
}$\U$ there is a unique $x$ with $\lim\U=\lim\{x\}^{\uparrow}$.
The two notions (sober and weakly sober) coincide for topological
spaces, but not for convergence spaces. Though \cite[Remark 1.9]{hoffmann1977irreducible}
mentions this, it provides neither a proof nor an example and does
not explore the conditions on a convergence space for weak sobriety
and sobriety to coincide. We fill this gap here (Example \ref{exa:weaklysobernotsober},
Proposition \ref{prop:weaklysoberimpliesober}).

In this context, it is useful to point out that the condition 
\begin{equation}
\tag{\ensuremath{S_{0}}}\left(x\in\lim\F\text{ and }t\in\lim\{x\}^{\uparrow}\right)\then t\in\lim\F\label{eq:S0}
\end{equation}
 is automatically true in a topological space because in this case
$t\in\lim\{x\}^{\uparrow}$ implies $\O(t)\subset\O(x)$. However
condition (\ref{eq:S0}) is not generally true in an arbitrary convergence
space. This (vacuous for topological spaces) weak diagonal axiom is
called $S_{0}$ in \cite{Schroder74}. 

Recall (e.g., \cite{DM.book}) that a convergence space is $T_{0}$
if points can be distinguished by the convergence structure, that
is, if
\[
x\neq t\then\left\{ \F\in\mathbb{F}X:x\in\lim\F\right\} \neq\left\{ \F\in\mathbb{F}X:t\in\lim\F\right\} ,
\]
and $T_{1}$ if singletons are closed. Of course, a topological space
is $T_{0}$ (resp. $T_{1}$) in the usual sense if and only if it
is in the convergence sense.

Note that if a convergence has closed limits (that is, $\lim\F$ is
closed for every filter) then it is $S_{0}$, but not conversely.
Of course, in a topological space limits are closed. It might be useful
to have Figure \ref{fig:weakdiagonalconditions} below in mind when
considering the next few results:

\begin{figure}[H]
\[
\xymatrix{ & \text{topological}\ar[r]\ar[d] & \text{diagonal\ar[d]}\\
 & \text{epitopological}\ar[d] & \forall\F\in\mathbb{F}X\;\adh\F\text{ closed\ar[d]}\\
T_{1}\ar[dd] & \forall\F\in\mathbb{F}X\;\text{\ensuremath{\lim\F\text{ closed}}}\ar[d]\ar@/_{2pc}/[ddl]\ar[r] & \forall\U\in\mathbb{U}X\;\text{\ensuremath{\lim\U\text{ closed}}}\ar^{\mathrm{pseudotop.}}@/_{2pc}/[l]\\
 & \forall\F\in\mathbb{F}X\;\lim\F\in\F\then\lim\F\text{ closed}\ar[d]\\
S_{0}\ar[dr] & \forall\U\in\mathbb{U}X\;\lim\U\in\U\then\lim\U\text{ closed}\ar[d]\ar_{\mathrm{pseudotop.}}@/_{1pc}/[u]\\
 & \forall x\in X,\lim\{x\}^{\uparrow}\text{ closed}
}
\]

\label{fig:weakdiagonalconditions}\caption{weak diagonality and separation conditions}
\end{figure}
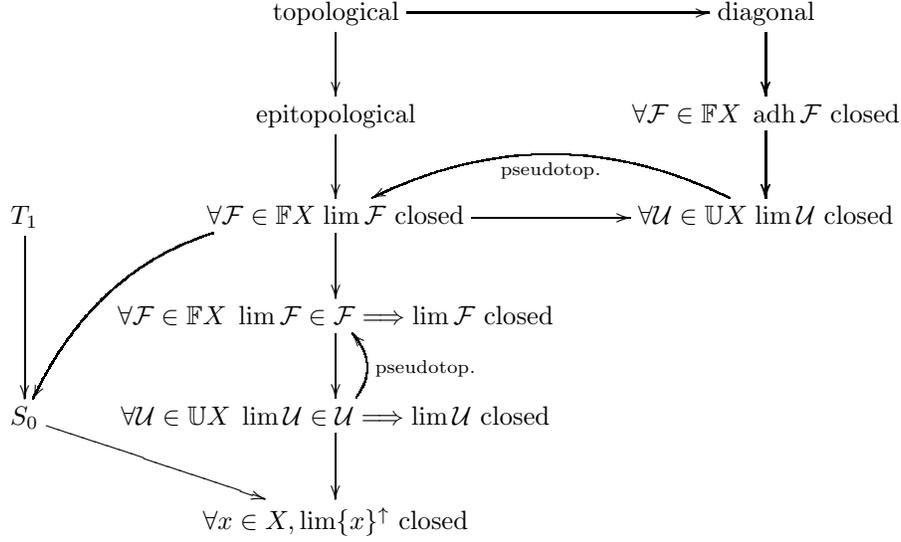

\begin{lem}
\label{lem:equallimpoints} 
\begin{enumerate}
\item If limit sets of principal ultrafilters are closed (in particular
if the convergence is $S_{0}$), then 
\[
\{x,y\}\subset\lim\{x\}^{\uparrow}\cap\lim\{y\}^{\uparrow}\iff\lim\{x\}^{\uparrow}=\lim\{y\}^{\uparrow}.
\]
\item If $X$ is $T_{0}$ and $S_{0}$ then 
\begin{equation}
\{x,y\}\subset\lim\{x\}^{\uparrow}\cap\lim\{y\}^{\uparrow}\then x=y.\label{eq:antisymmetric}
\end{equation}
\end{enumerate}
\end{lem}

\begin{proof}
It is clear that if $\lim\{x\}^{\uparrow}=\lim\{y\}^{\uparrow}$ then
in particular $\{x,y\}\subset\lim\{x\}^{\uparrow}\cap\lim\{y\}^{\uparrow}$.
Conversely, if $\lim\{x\}^{\uparrow}$ is closed and belongs to $\{y\}^{\uparrow}$,
then $\lim\{y\}^{\uparrow}\subset\lim\{x\}^{\uparrow}$ and similarly
for the reverse inclusion if $\lim\{y\}^{\uparrow}$ belongs to $\{x\}^{\uparrow}$.

For the second part if $x\ne y$, there is a filter $\F$ with $\card(\lim\F\cap\{x,y\})=1$
because the convergence is $T_{0}$. Say $x\in\lim\F.$ Then $y\notin\lim\{x\}^{\uparrow}$
for otherwise $y\in\lim\F$ by $S_{0}$, and similarly $x\notin\lim\{y\}^{\uparrow}$
if $x\in\lim\F$, so that $\{x,y\}\not\subset\lim\{x\}^{\uparrow}\cap\lim\{y\}^{\uparrow}$. 
\end{proof}
In view of Lemma \ref{lem:equallimpoints}, in a $T_{0}$ \emph{topological}
space a generic point of an irreducible closed set is necessarily
unique. However, 

\begin{equation}
\lim\{x\}^{\uparrow}=\lim\{y\}^{\uparrow}\then x=y\label{eq:equalgeneric}
\end{equation}
requires both $S_{0}$ and $T_{0}$ for general convergence spaces
as shown by Examples \ref{exa:multigeneric} and \ref{exa:multigenericS0}
below.
\begin{prop}
\label{prop:weaklysoberimpliesober} A weakly sober convergence space
in which limits of irreducible filters are closed (in particular,
a topological space) is sober.
\end{prop}

\begin{proof}
If $X$ is a weakly sober convergence space and $\F$ is an irreducible
filter on $X$, then an ultrafilter $\U$ finer than $\F$ is also
irreducible, so that $\lim\U=\lim\{x_{\U}\}^{\uparrow}$ for a unique
point $x_{\U}\in X$. Moreover, $\lim\F\subset\lim\U$ because $\U\geq\F$
and as $\lim\F\in\F\subset\U$ and $\lim\F$ is closed, $\lim\U\subset\lim\F$.
Hence $\lim\U=\lim\F$ for every $\U\in\mathbb{U}(\F)$. Hence all
points $x_{\U}$ coincide and this point is the unique generic point
for $\F$.
\end{proof}
It turns out that all new phenomena in the realm of general convergences
(as opposed to topological spaces) can be observed with finite examples.
Filters on a finite sets are all principal and a finitely deep convergence
on a finite set is entirely determined by its restriction to principal
ultrafilters. Hence such convergences can easily be described by diagrams
in which $x\to y$ means that $y\in\lim\{x\}^{\uparrow}$ and we include
such pictures in each relevant example, using this convention and
systematically omitting the loops (i.e., we do not depict the automatic
convergence relation $x\in\lim\{x\}^{\uparrow}$ ).
\begin{example}[A $T_{0}$ convergence with $x\neq y$ and \textbf{$\lim\{x\}^{\uparrow}=\lim\{y\}^{\uparrow}$}]
 \label{exa:multigeneric} Let $X=\{x,y,z\}$ with the convergence
of finite depth given by $\lim\{x\}^{\uparrow}=\lim\{y\}^{\uparrow}=\lim\{x,y\}^{\uparrow}=\{x,y\}$
and $\lim\{z\}^{\uparrow}=\{z,x\}$. By finite depth, $\lim\{x,y,z\}^{\uparrow}=\lim\{y,z\}^{\uparrow}=\lim\{x,z\}^{\uparrow}=\{x\}$.
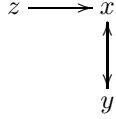
\[
\xymatrix{z\ar[r] & x\ar[d]\\
 & y\ar[u]
}
\]
\end{example}

\begin{example}[A (non $T_{0}$) $S_{0}$ convergence space with $x\neq y$ and \textbf{$\lim\{x\}^{\uparrow}=\lim\{y\}^{\uparrow}$} ]
\label{exa:multigenericS0} Let $X=\{x,y,z\}$ the convergence of
finite depth given by $\lim\{x\}^{\uparrow}=\lim\{y\}^{\uparrow}=\lim\{x,y\}^{\uparrow}=\{x,y\}$
and $\lim\{z\}^{\uparrow}=\{z\}$. By finite depth, $\lim\{x,y,z\}^{\uparrow}=\lim\{y,z\}^{\uparrow}=\lim\{x,z\}^{\uparrow}=\emptyset$.
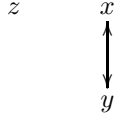
\[
\xymatrix{z & x\ar[d]\\
 & y\ar[u]
}
\]
\end{example}

Note that weakly sober implies (\ref{eq:equalgeneric}): if there
are points $x\neq y$ with $\lim\{x\}^{\uparrow}=\lim\{y\}^{\uparrow}$
then $\{x\}^{\uparrow}$ is an irreducible ultrafilter with two different
generic points, so that the space is not weakly sober. Hence the easy
fact that finite $T_{0}$ topological spaces are sober fails to extend
to convergence spaces, because of the uniqueness requirement on generic
points, though this requirement is vacuous in topological spaces.
Therefore, it makes sense to also consider the variants without this
requirement:

A convergence space $X$ is \emph{quasi-sober }(resp. \emph{weakly
quasi-sober}) if for every irreducible filter $\F$ (respectively,
for every irreducible ultrafilter $\F$) on $X$ there is a (not necessarily
unique) point $x$ of $X$ with $\lim\F=\lim\{x\}^{\uparrow}$. Example
\ref{exa:multigeneric} is $T_{0}$ and quasi-sober but not weakly
sober. Example \ref{exa:multigenericS0} is $S_{0}$ and quasi-sober
but not weakly sober. A quasi-sober space does not need to be $T_{0}$
for the antidiscrete topology is quasi-sober. Of course, if a convergence
space satisfies (\ref{eq:equalgeneric}) and is quasi weakly sober,
it is weakly sober. Lemma \ref{lem:equallimpoints} states that this
is the case in a $T_{0}$ space with closed limits of principal ultrafilters.
More generally, let us call a convergence space \emph{antisymmetric}
if (\ref{eq:antisymmetric}) holds and \emph{almost antisymmetric},
or \emph{aas} for short, if (\ref{eq:equalgeneric}) holds. Note that:
\begin{prop}
\label{prop:wsoberandqwsober} 
\begin{enumerate}
\item A convergence space is weakly sober if and only if it is weakly quasi-sober
and aas. 
\item An antisymmetric space is $T_{0}$.
\item In a convergence in which limits of principal ultrafilters are closed,
the following are equivalent:
\begin{enumerate}
\item weakly sober;
\item weakly quasi-sober and antisymmetric;
\item weakly quasi-sober and aas.
\end{enumerate}
\end{enumerate}
In particular a weakly sober space in which limits of principal ultrafilters
are closed is $T_{0}$.
\end{prop}

\begin{proof}
(1) If the space is weakly sober and $\lim\{x\}^{\uparrow}=\lim\{y\}^{\uparrow}$
then $x=y$ by uniqueness of generic points for limits of irreducible
ultrafilters, hence the space is aas. Conversely, aas ensure uniqueness
of generic points in a weakly quasi-sober space. 

(2) Suppose $\{\F\in\mathbb{F}X:x\in\lim\F\}=\{\F\in\mathbb{F}X:y\in\lim\F\}$.
In particular $\{x,y\}\subset\lim\{x\}^{\uparrow}\cap\lim\{y\}^{\uparrow}$
so that $x=y$ by aas.

For (3), to see $(a)\then(b)$, suppose $\{x,y\}\subset\lim\{x\}^{\uparrow}\cap\lim\{y\}^{\uparrow}$.
Then $\lim\{x\}^{\uparrow}\subset\lim\{y\}^{\uparrow}$ because $\lim\{y\}^{\uparrow}\in\{x\}^{\uparrow}$
and $\lim\{y\}^{\uparrow}$ is closed, and similarly for the reverse
inclusion. $(b)\then(c)$ is clear and $(a)\iff(c)$ is (1). 
\end{proof}

A topological sober space is always $T_{0}$ but things are slightly
more complicated for the general case:
\begin{example}[A sober convergence that is not $T_{0}$, whose modification of finite
depth is not sober]
\label{exa:finitedepthnotsober} Let $X=\{x,y,s,t\}$ with the convergence
given by $\lim\{x\}^{\uparrow}=\{x,y,s\}$, $\lim\{y\}^{\uparrow}=\{x,y,t\}$,
$\lim\{s\}^{\uparrow}=\{s\}$ and $\lim\{t\}^{\uparrow}=\{t\},$ and
all other filters do not converge (note that this convergence is not
of finite depth).
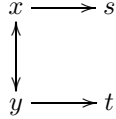
\[
\xymatrix{x\ar[d]\ar[r] & s\\
y\ar[u]\ar[r] & t
}
\]
 Then $\lim^{-1}(x)=\lim^{-1}(y)=\{\{x\}^{\uparrow},\{y\}^{\uparrow}\}$
but $\lim\{x\}^{\uparrow}\neq\lim\{y\}^{\uparrow}$ and the condition
of sobriety is satisfied. Note that the modification of finite depth
$\operatorname{L}\xi$ would add the irreducible filter $\{x,y\}^{\uparrow}$
because the $\lim_{\operatorname{L}\xi}\{x,y\}^{\uparrow}=\lim_{\xi}\{x\}^{\uparrow}\cap\lim_{\xi}\{y\}^{\uparrow}=\{x,y\}$
but this limit has no generic point, hence sobriety would fail. 
\end{example}

In other words, the modification of finite depth of a sober convergence
may fail to be sober. On the other hand, the example above is easily
modified to make it of a finite depth:
\begin{example}[A sober convergence of finite depth that is not $T_{0}$]
\label{exa:soberfinitedepthnotT0} Let $X=\{x,y,s,t\}$ with the
convergence of finite depth given by $\lim\{x\}^{\uparrow}=\{x,y,z,s\}$,
$\lim\{y\}^{\uparrow}=\{x,y,z,t\}$, $\lim\{z\}^{\uparrow}=\{x,y,z\}$,
$\lim\{s\}^{\uparrow}=\{s\}$ and $\lim\{t\}^{\uparrow}=\{t\}.$ 
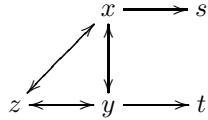
\[
\xymatrix{ & x\ar[d]\ar[r]\ar[dl] & s\\
z\ar[r]\ar[ru] & y\ar[u]\ar[r]\ar[l] & t
}
\]

This convergence is not $T_{0}$ as the same filters converge to $x,y$
and $z$. By finite depth, 
\[
\lim\{x,y\}^{\uparrow}=\lim\{x,z\}^{\uparrow}=\lim\{y,z\}^{\uparrow}=\lim\{x,y,z\}^{\uparrow}=\{x,y,z\},
\]
so that these four filters are irreducible and all have the same unique
generic point $z$.
\end{example}

\begin{example}[A $T_{0}$ weakly sober convergence of finite depth that is not sober]
 \label{exa:weaklysobernotsober} On $X=\{x,y,t,s,w,z\}$ consider
the convergence of finite depth given by $\lim\{x\}^{\uparrow}=\{x,y,s,t\}$,
$\lim\{y\}^{\uparrow}=\{x,y,t,w\}$, $\lim\{t\}^{\uparrow}=\{t\}$,
$\lim\{s\}^{\uparrow}=\{s\}$, $\lim\{w\}^{\uparrow}=w$ and $\lim\{z\}^{\uparrow}=\{z,x\}$. 

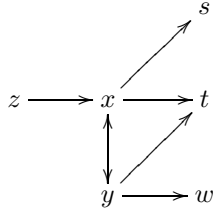
\[
\xymatrix{ &  & s\\
z\ar[r] & x\ar[ur]\ar[d]\ar[r] & t\\
 & y\ar[u]\ar[r]\ar[ur] & w
}
\]

This is a $T_{0}$ weakly sober convergence of finite depth that is
not sober, because $\{x,y\}^{\uparrow}$ is irreducible but does not
have a generic point.
\end{example}

Note that if $\xi\geq\tau$ and $\tau$ is (weakly) sober then $\xi$
is (weakly) sober. In particular, if $\T\xi$ is (weakly) sober, so
is $\xi$ but the converse is false:
\begin{example}[A $T_{0}$ sober convergence of finite depth with a non-sober topological
modification]
 \label{exa:sobertopmodnotsober} Take on $X=\{x,y,z\}$ the finitely
deep convergence given by $\lim\{x\}^{\uparrow}=\{x,y\}$, $\lim\{y\}^{\uparrow}=\{x,y,z\}$
and $\lim\{z\}^{\uparrow}=\{y,z\}$. Then by finite depth, $\lim\{x,y\}^{\uparrow}=\{x,y\}$,
$\lim\{x,z\}^{\uparrow}=\lim\{x,y,z\}^{\uparrow}=\{y\}$, and $\lim\{y,z\}^{\uparrow}=\{y,z\}$. 

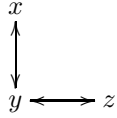
\[
\xymatrix{x\ar[d]\\
y\ar[u]\ar[r] & z\ar[l]
}
\]
The irreducible filters are $\{x\}^{\uparrow},\{y\}^{\uparrow},\{z\}^{\uparrow}$
,$\{x,y\}^{\uparrow}$ and $\{y,z\}^{\uparrow}$, $x$ is the only
generic point for $\lim\{x\}^{\uparrow}$ and for $\lim\{x,y\}^{\uparrow}$,
$y$ is the only generic point for $\lim\{y\}^{\uparrow}$, $z$ is
the only generic point for $\lim\{z\}^{\uparrow}$ and for $\lim\{y,z\}^{\uparrow}$
. Hence this convergence is sober, but its topological modification
is antidiscrete, hence is not even $T_{0}$.
\end{example}

As observed in \cite{hoffmann1977irreducible}, an arbitrary product
of (weakly) sober convergence spaces is (weakly) sober and a subspace
of a (weakly) sober convergence does not need to be (weakly) sober,
though a closed subspace does. 

\section{More on the convergence space $\protect\pt L$}

\begin{prop}
\label{prop:ptLweaklysober} For every convergence lattice $L$ the
convergence space $\pt L$ is weakly quasi-sober.
\end{prop}

\begin{proof}
Suppose $\lim_{\pt L}\U\in\U$ for some $\U\in\mathbb{U}(\pt L)$.
Then $\U^{\circ}\in\pt L$ and moreover $\lim_{\pt L}\{\U^{\circ}\}^{\uparrow}=\lim_{\pt L}\U$.
Indeed, $\U^{\circ}$ is prime for if $\ell\vee m\in\U^{\circ}$ then,
since elements of $\pt L$ are prime filters, $(\ell\vee m)^{\bullet}=\ell^{\bullet}\cup m^{\bullet}\in\U$
and $\U$ is an ultrafilter so either $\ell^{\bullet}$ or $m^{\bullet}$
belongs to $\U$, that is, either $\ell\in\U^{\circ}$ or $m\in\U^{\circ}$.
Moreover, as $\lim_{\pt L}\U=(\lim_{\xi}\U^{\circ})^{\bullet}\in\U$
we conclude that $\lim_{\xi}\U^{\circ}\in\U^{\circ}$, by definition
of $\U^{\circ}$. The fact that $\lim_{\pt L}\{\U^{\circ}\}^{\uparrow}=\lim_{\pt L}\U$
follows from Lemma \ref{lem:principalofpoints}.
\end{proof}

\begin{prop}
\label{prop:ptLaas} Let $(L,\lim_{L})$ be a convergence lattice
satisfying 
\begin{equation}
\lm_{L}x\wedge\lm_{L}y\in x\cap y\then x=y\label{eq:aasinL}
\end{equation}
for every $x,y\in\pt L$. Then the convergence space $\pt L$ is aas,
hence weakly sober.
\end{prop}

\begin{proof}
Assume (\ref{eq:aasinL}) and suppose $x,y\in\pt L$ with $\lim_{\pt L}\{x\}^{\uparrow}=\lim_{\pt L}\{y\}^{\uparrow}$,
that is,
\[
\left(\lm_{L}(\{x\}^{\uparrow})^{\circ}\right)^{\bullet}=\left(\lm_{L}(\{y\}^{\uparrow})^{\circ}\right)^{\bullet}.
\]
In view of Lemma \ref{lem:principalofpoints}, $(\lim_{L}x)^{\bullet}=(\lim_{L}y)^{\bullet}$
so that $\lm_{L}x\wedge\lm_{L}y\in x\cap y$. As a result of (\ref{eq:aasinL}),
$x=y$.

In view of Propositions \ref{prop:ptLweaklysober} and \ref{prop:wsoberandqwsober},
$\pt L$ is not only quasi weakly sober but also weakly sober.
\end{proof}
An element $\ell$ of a convergence lattice $(L,\lim)$ is \emph{closed
}(e.g., \cite{myn.ptfreeAP} \footnote{called \emph{quasi-closed }in \cite{FredetJean}.})
if $\lim\F\leq\ell$ whenever $\ell\in\F$ (equivalently whenever
$\ell\in\F^{\#}$ where $\F^{\#}=\{m\in L:\forall f\in\F\;m\wedge f>\bot\}$)
and \emph{open }if $\lim\F\wedge\ell>\bot$ implies $\ell\in\F$ (\footnote{called \emph{fully open }in \cite{convframes} and a weaker notion
than the notion of open element in \cite{FredetJean}}). This is an abstraction of the case $L=(\mathbb{P}X,\lim_{\xi})$
in which closed elements are exactly the closed subsets of $(X,\xi)$
and open elements are exactly the open subsets. Let $\O_{L}$ and
$\C_{L}$ denote the sublattices of $L$ formed by open elements and
closed elements respectively.

\begin{lem}
\label{lem:closedelmttosubset} If $\ell\in L$ is closed then $\ell^{\bullet}$
is a closed subset of $\pt L$. If $\ell\in L$ is open, then $\ell^{\bullet}$
is an open subset of $\pt L$.
\end{lem}

\begin{proof}
Let $x\in\lim_{\pt L}\F$ with $\ell^{\bullet}\in\F$. We need to
show that $x\in\ell^{\bullet}$, equivalently, that $\ell\in x$.
Since $\lim_{L}\F^{\circ}\in x$ and $\ell\in\F^{\circ}$ for $\ell$
closed, we conclude that $\lim_{L}\F^{\circ}\leq\ell$, so that $\ell\in x$.

Let $\F\in\mathbb{FP}\pt L$ with $x\in\lim_{\pt L}\F\cap\ell^{\bullet}$,
that is, $\ell\in x$ and $\lim_{L}\F^{\circ}\in x$ so that $\ell\wedge\lim_{L}\F^{\circ}\neq\bot$
and $\ell$ is open, hence $\ell\in\F^{\circ}$, that is, $\ell^{\bullet}\in\F$.
\end{proof}
\begin{prop}
\label{prop:limLpointclosed} If $(L,\lim_{L})$ is a convergence
lattice in which
\begin{enumerate}
\item $\lim_{L}\F$ is closed for every $\F\in\mathbb{F}L$ then limits
sets are closed in the convergence space $\pt L$. 
\item $\lim_{L}x$ is closed for every $x\in\pt L$ then $\lim_{\pt L}\{x\}^{\uparrow}$
is closed for every $x\in\pt L$.
\end{enumerate}
\end{prop}

\begin{proof}
Let $\lim_{\pt L}\F\in\G\in\mathbb{FP}(\pt L)$. By definition $\lim_{\pt L}\F=\left(\lim_{L}\F^{\circ}\right)^{\bullet}$
is closed because $\lim_{L}\F^{\circ}$ is closed, by Lemma \ref{lem:closedelmttosubset}.
The second part is the particular case where $\F=\{x\}^{\uparrow}$
for $x\in\pt L$, with the observation that in this case $\F^{\circ}=x$
by Lemma \ref{lem:principalofpoints}.
\end{proof}
\begin{cor}
\label{cor:ptLweaklysober} If $(L,\lim_{L})$ is a convergence lattice
in which $\lim_{L}x$ is closed whenever $x\in\pt L$, the following
are equivalent:
\begin{enumerate}
\item $\pt L$ is weakly sober;
\item $\pt L$ is antisymmetric; 
\item $\pt L$ is aas;
\item (\ref{eq:aasinL}) for every $x,y\in\pt L$;
\item $\lim_{L}x=\lim_{L}y\then x=y$ for every $x,y\in\pt L$.
\end{enumerate}
\end{cor}

\begin{proof}
By Propositions \ref{prop:ptLweaklysober} and \ref{prop:limLpointclosed},
$\pt L$ is a weakly quasi-sober space in which limits of principal
ultrafilters are closed. In view of Proposition \ref{prop:wsoberandqwsober},
(1), (2) and (3) are equivalent. That $(4)\then(5)$ is clear and
$(4)\then(3)$ is Proposition \ref{prop:ptLaas}. To see $(5)\then(4)$
suppose that $\lm_{L}x\wedge\lm_{L}y\in x\cap y$. Then $\lim_{L}y\leq\lim_{L}x$
because $\lim_{L}x$ is closed and belongs to $y$, and similarly
for the reverse inequality. Hence by (5), $x=y$.

To see that $(1)\then(5)$ suppose $\lim_{L}x=\lim_{L}y$ so that,
in view of Lemma \ref{lem:principalofpoints}, $\lim_{L}(\{x\}^{\uparrow})^{\circ}=\lim_{L}(\{y\}^{\uparrow})^{\circ}$
so that $\lim_{\pt L}\{x\}^{\uparrow}=\lim_{\pt L}\{y\}^{\uparrow}$.
By uniqueness of generic point of irreducible ultrafilters in $\pt L$,
$x=y$.
\end{proof}
In the case $L=(\mathbb{P}X,\lim_{\xi})$, this means:
\begin{cor}
\label{cor:weaklysoberptP} If $(X,\xi)$ is a convergence space in
which limits of irreducible ultrafilters are closed then $(X,\xi)$
is a dense subspace of $\pt(\mathbb{P}X,\lim_{\xi})$ which is weakly
sober if and only if 
\begin{equation}
\lm_{\xi}\U=\lm_{\xi}\W\then\U=\W\label{eq:ptPaas}
\end{equation}
for every pair $\U,\W$ of irreducible ultrafilters. 
\end{cor}

We will see with Remark \ref{rem:notaas} below that the condition
(\ref{eq:ptPaas}) turns out to be very restrictive.

An alternative approach to force (\ref{eq:aasinL}) is to consider
the equivalence relation $\sim$ on $\pt L$ defined by $x\sim y$
if $\lim_{L}x=\lim_{L}y$ and consider the quotient set 
\[
\pt'L:=\pt L/\sim,
\]
endowed with the quotient convergence. To be explicit, with $q:\pt L\to\pt'L$
the canonical surjection, this means
\begin{eqnarray*}
\lm_{\pt'L}\F & = & \bigcup_{q[\G]\leq\F}q(\lm_{\pt L}\G)\\
 & = & \bigcup_{q[\G]\leq\F}q((\lm_{L}\G^{\circ})^{\bullet}).
\end{eqnarray*}

\begin{rem}
\label{rem:qislim} Note that we can identify $\pt'L$ with $\{\lim_{L}x:x\in\pt L\}\subset L$
and in this interpretation, the canonical surjection $q$ is the restriction
$\lim'_{L}:\pt L\to\pt'L$ of $\lim_{L}$ to points, when points of
$L$ are interpreted as filters on $L$. Hence we can consider the
order induced by $L$ on $\pt'L$ (\footnote{namely, if $x,y\in\pt'L$, let $x\sqsubseteq y$ denote $\lim_{L}t\leq\lim_{L}s$
where $t\in q^{-1}(x)$ and $s\in q^{-1}(y)$, which is well-defined
by definition of $\sim$, but with the interpretation 
\[
\pt'L=\{\lm_{L}x:x\in\pt L\}\subset L
\]
There is no need to distinguish $\sqsubseteq$ from the order $\leq$
of $L$.}) which we also denote $\leq$. Note that $\pt'L$ is not a sublattice
of $L$ in general. In contrast the subset $\C_{L}$ of closed elements
of $L$ is a sublattice of $L$.
\end{rem}

Note that 
\begin{equation}
\forall c\in\C_{L}\forall x\in\pt L\;\left(c\in x\iff\lm_{L}x\leq c\right),\label{eq:closedinpoint}
\end{equation}
because $c\in x\then\lim_{L}x\leq c$ follows from $c$ closed, and
the converse is true whether $c$ is closed or not, because $\lim_{L}x\in x$.
Hence $x\cap\C_{L}=(\uparrow\lim_{L}x)\cap\C_{L}$ whenever $x\in\pt L$. 
\begin{lem}
\label{lem:Lclosedlims} If $(L,\lim_{L})$ is a convergence lattice
with closed limits, that is, $\lim_{L}:\mathbb{F}L\to\C_{L}$, then
\[
x\in\lm_{\pt L}\F\iff\lm_{L}x\leq\lm_{L}\F^{\circ},
\]
that is, 
\begin{equation}
\lm_{\pt L}\F=q^{-1}(\pt'L\cap\downarrow\lm_{L}\F^{\circ}).\label{eq:ptLlimwithq}
\end{equation}

Thus
\begin{equation}
\lm_{\pt'L}\H=\bigcup_{q[\F]\leq\H}\pt'L\cap\downarrow\lm_{L}\F^{\circ}.\label{eq:ptL'}
\end{equation}

In particular, $\lim_{\pt'L}\H=\downarrow\lim_{\pt'L}\H$.
\end{lem}

\begin{proof}
By definition, $x\in\lim_{\pt L}\F$ if and only if $\lim_{L}\F^{\circ}\in x$,
which is equivalent to $\lim_{L}x\leq\lim_{L}\F^{\circ}$ by (\ref{eq:closedinpoint}).
The formula (\ref{eq:ptLlimwithq}) is a rephrasing from which (\ref{eq:ptL'})
follows because $q$ is onto and thus $qq^{-1}=\id_{\pt L'}$.
\end{proof}
In particular, if $\H=\{\lim_{L}x\}^{\uparrow}$ is a principal ultrafilter
on $\pt'L$ then: 
\begin{cor}
\label{cor:pt'aas} If $(L,\lim_{L})$ is a convergence lattice with
closed limits, then
\begin{equation}
\lm_{\pt'L}\{\lm_{L}x\}^{\uparrow}=\pt'L\cap\downarrow\lm_{L}x.\label{eq:limsingletoninpt'L}
\end{equation}
Thus $\pt'L$ is an aas convergence space.
\end{cor}

\begin{proof}
Note that $\lm_{\pt'L}\{\lm_{L}x\}^{\uparrow}=\bigcup_{q[\F]\leq\{\lim_{L}x\}^{\uparrow}}\pt'L\cap\downarrow\lm_{L}\F^{\circ}$
by Lemma \ref{lem:Lclosedlims} and that $q[\F]\leq\{\lim_{L}x\}^{\uparrow}$
means $(\lim_{L}x)^{\bullet}\in\F^{\#}$ because $\lim_{L}x\in q[F]$
if there is $y\in F$ with $\lim_{L}y=\lim_{L}x$ and thus $y\in(\lim_{L}x)^{\bullet}\cap F$.
As $\lim_{L}x$ is a closed element of $L$, $(\lim x)^{\bullet}$
is a closed subset of $\pt L$ by Lemma \ref{lem:closedelmttosubset},
so that $\lim_{\pt L}\F=(\lim_{L}\F^{\circ})^{\bullet}\subset(\lim_{L}x)^{\bullet}$.
Hence if $\lim_{L}y\in\pt'L\cap\downarrow\lim_{L}\F^{\circ}$ then
$y\in(\lim_{L}\F^{\circ})^{\bullet}\subset(\lim_{L}x)^{\bullet}$,
that is, $\lim_{L}x\in y$ and $\lim_{L}x$ is closed thus $\lim_{L}y\leq\lim_{L}x$.
Consequently, $\lim_{L}y\in\pt'L\cap\downarrow\lim_{L}x$. Conversely,
if $\lim_{L}y\leq\lim_{L}x$ then $\lim_{L}x\in y$ and thus $\F=\{y\}^{\uparrow}$
satisfies $q[\F]\leq\{\lim_{L}x\}^{\uparrow}$ and $\lim_{L}y=\lim_{L}\F^{\circ}$
because $\F^{\circ}=y$ by Lemma \ref{lem:closedelmttosubset}, which
completes the proof of (\ref{eq:limsingletoninpt'L}).

As a result of (\ref{eq:limsingletoninpt'L}), if $\lm_{\pt'L}\{\lm_{L}x\}^{\uparrow}=\lm_{\pt'L}\{\lm_{L}y\}^{\uparrow}$
then $\lim_{L}x\leq\lim_{L}y$ and $\lim_{L}y\leq\lim_{L}x$ so that
$\lim_{L}x=\lim_{L}y$. Hence, $\pt'L$ is aas.
\end{proof}
\begin{lem}
\label{lem:saturatedoepnclosed} If $U$ is an open or a closed subset
of $\pt L$ then $U$ is saturated, that is, $U=q^{-1}(q(U))$. In
particular, if $\ell$ is either an open or a closed element of $L$
then $\ell^{\bullet}$ is saturated. 
\end{lem}

\begin{proof}
Suppose $U$ is open and $x\in q^{-1}(q(U))$ that is, $\lim_{L}x\in q(U)$,
that is, there is $t\in U$ with $\lim_{L}t=\lim_{L}x$. As 
\[
\lm_{\pt L}\{t\}^{\uparrow}=(\lm_{L}(\{t\}^{\uparrow})^{\circ})^{\bullet}=(\lm_{L}t)^{\bullet}=(\lm_{L}x)^{\bullet}=\lm_{\pt L}\{x\}^{\uparrow},
\]
we conclude that $\lim_{\pt L}\{x\}^{\uparrow}\cap U\neq\emptyset$
so that $U\in\{x\}^{\uparrow}$ because $U$ is open, that is, $x\in U$.
If $U$ is closed instead, then $\lim_{\pt L}\{t\}^{\uparrow}=\lim_{\pt L}\{x\}^{\uparrow}\subset U$
and we conclude similarly. The case of $\ell^{\bullet}$ follows from
Lemma \ref{lem:closedelmttosubset}.
\end{proof}

\begin{cor}
\label{cor:qopenclosed} $q=\lim'_{L}:\pt L\to\pt'L$ is an open and
closed map (sends open sets to open sets and closed sets to closed
sets). Moreover, the initial topology for $\lim_{L}'$ is the topological
modification of $\pt L$:
\[
\T(\pt L)=q^{-}(\T\pt'L)=q^{-}(\T q(\pt L)).
\]
\end{cor}

\begin{lem}
\label{lem:cbulletcompl} If $c$ is a closed element of $(L,\lim_{L})$
then 
\[
q^{-1}(\pt'L\setminus\downarrow c)=\pt L\setminus c^{\bullet}.
\]
\end{lem}

\begin{proof}
$x\in q^{-1}(\pt'L\setminus\downarrow c)$ if and only if $\lim_{L}x\nleq c$,
which, in view of (\ref{eq:closedinpoint}), is equivalent to $c\notin x$,
that is, to $x\in\pt L\setminus c^{\bullet}$.
\end{proof}
Note that if $L=(\mathbb{P}X,\lim_{\xi})$ then for every open subset
$U$ of $\pt L$ there is an open element $\ell$ of $L$ with $U=\ell^{\bullet}$
and similarly for closed sets. Indeed:

\begin{prop}
\label{prop:UonXbullet-1} Let $L=(\mathbb{P}X,\lim_{\xi})$ . Every
subset $U$ of $\pt L$ that is either open or closed satisfies $U=(U\cap X)^{\bullet}$.
\end{prop}

\begin{proof}
Assume that $U$ is open. Let $\U\in(U\cap X)^{\bullet}$, that is,
$(U\cap X)\in\U$. As $\lm_{\pt L}\{\U\}^{\uparrow}=(\lm_{\xi}(\{\U\}^{\uparrow})^{\circ})^{\bullet}=(\lm_{\xi}\U)^{\bullet}$
and there is $x\in\lim_{\xi}\U\cap U\cap X$, we conclude that $\{x\}^{\uparrow}\in(\lim_{\xi}\U)^{\bullet}\cap U$.
Hence $\lm_{\pt L}\{\U\}^{\uparrow}\cap U\neq\emptyset$ and $U$
is $\pt L$-open so that $U\in\{\U\}^{\uparrow}$, that is, $\U\in U$.
Conversely, if $\U\in U$, since $\U\in\lim_{\pt L}h[\U]$ where $h:X\to\pt L$
is $h(x)=\{x\}^{\uparrow}$ and $U$ is open, $U\cap X\in h[\U]$,
that is, $\U\in(U\cap X)^{\bullet}$.

Assume now that $U$ is closed. If $\U\in U$ then $\lm_{\pt L}\{\U\}^{\uparrow}=(\lm_{\xi}\U)^{\bullet}\subset U$
and thus $\lim_{\xi}\U\subset U\cap X$ so that $U\cap X\in\U$, that
is, $\U\in(U\cap X)^{\bullet}$. Conversely, if $\U\in(U\cap X)^{\bullet}$,
that is, $U\cap X\in\U$, then $\U\in\lim_{\pt L}h[\U]\subset U$
so that $\U\in U$.
\end{proof}
Let us say that a convergence lattice $(L,\lim_{L})$ \emph{has enough
closed elements }if for every closed subset $C$ of $\pt L$ there
is a closed element $c\in L$ with $C=c^{\bullet}$. We similarly
say that $(L,\lim_{L})$ \emph{has enough open elements} if every
open subset $O$ of $\pt L$ is of the form $u^{\bullet}$ for some
open element $u$ of $L$. By Proposition \ref{prop:UonXbullet-1},
$L=(\mathbb{P}X,\lim_{\xi})$ always has enough closed elements and
enough open elements.
\begin{thm}
\label{thm:uppertopology} If $(L,\lim_{L})$ is a convergence lattice
with enough closed elements, then the topological modification of
$\pt'L$ is given by the open sets
\[
\left\{ \pt L'\setminus\downarrow c:c\in\C_{L}\right\} .
\]
\end{thm}

\begin{proof}
In view of Lemma \ref{lem:cbulletcompl}, sets of the form $\pt L'\setminus\downarrow c$
for a closed element $c\in L$ are $\pt'L$-open, because their preimages
under $q$ are $\pt L$-open by Lemma \ref{lem:cbulletcompl}. If
now $V$ is $\pt'L$-open, that is, $q^{-1}(V)$ is $\pt L$-open,
then $\pt L\setminus q^{-1}(V)$ is closed, hence there is a closed
element $c$ of $L$ with $q^{-1}(V)=\pt L\setminus c^{\bullet}$.
In view of Lemma \ref{lem:cbulletcompl}, $q^{-1}(V)=q^{-1}(\pt'L\setminus\downarrow c)$
and $q$ is surjective, so that $V=\pt'L\setminus\downarrow c$.
\end{proof}
\begin{rem}
Recall that the \emph{upper topology of a poset $(X,\leq)$ }is that
in which sets $\{X\setminus\downarrow x:x\in X\}$ form a subbase
of open sets (e.g., \cite{MR3086734}). Note that \emph{when $L$
has closed limits} (that is, $\lim_{L}:\mathbb{F}L\to\C_{L}$), then
$\pt'L\subset\C_{L}$ and the topology induced on $\pt'L$ by the
upper topology of $\C_{L}$ is given by the subbase $\left\{ \pt L'\setminus\downarrow c:c\in\C_{L}\right\} $,
so that, in view of Theorem \ref{thm:uppertopology}, it coincides
with the quotient topology induced by $\lim'_{L}:\pt L\to\pt'L$ where
$\pt L$ carries the topological modification of its standard convergence.
\end{rem}

\section{$\protect\Z$-regularity and structure of $\protect\pt'L$}

If $\Z\subset L$ we say that $(L,\lim_{L})$ is $\Z$-\emph{regular
}if 
\[
\lm_{L}\F=\lm_{L}\uparrow(\F\cap\Z)
\]
for every $\F\in\mathbb{F}L$. This is a generalization of the notion
defined for convergence spaces in \cite[Section VII.2]{DM.book},
namely, if $\Z\subset\mathbb{P}X$, then $(X,\xi)$ is $\Z$-\emph{regular
}if $\lim_{\xi}\F=\lim_{\xi}(\F\cap\Z)^{\uparrow}$ for every filter
$\F\in\mathbb{FP}X$.

Note that a convergence space $(X,\xi)$ is $\O_{\xi}$-regular if
and only if it is topological (\footnote{Indeed, it is clear that a topological space is $\O_{\xi}$-regular,
and conversely, if $\xi$ is $\O_{\xi}$-regular, then 
\[
x\in\lim\{x\}^{\uparrow}=\lim\left(\{x\}^{\uparrow}\cap\O_{\xi}\right)^{\uparrow}=\lim\N_{\xi}(x),
\]
so that $\xi$ is topological.}), equivalently, the convergence lattice $L=(\mathbb{P}X,\lim_{\xi})$
is $\O_{L}$-regular. 

Note that:
\begin{prop}
\label{prop:openimageoftop} Let $f:(X,\xi)\to(Y,\tau)$ be a continuous,
onto, and open map. If $\xi$ is topological, so is $\tau$.
\end{prop}

\begin{proof}
We show $\tau$ is $\O$-regular. To this end, let $y\in\lim_{\tau}\G$.
In view of \cite[Exercise XV.1.29]{DM.book}, $\G\geq f[\N_{\xi}(x)]$
for every $x\in f^{-}(y)$. Because $f$ is open $f[\N_{\xi}(x)]\subset(\G\cap\O_{\tau})^{\uparrow}$
and $y\in\lim_{\tau}f[\N_{\xi}(x)]$ by continuity. Hence, $y\in\lim_{\tau}(\G\cap\O_{\tau})^{\uparrow}$.
\end{proof}
Similarly a convergence space $(X,\xi)$ is topologically regular
(in the sense of \cite[Section VI.4]{DM.book}) if and only if it
is $\C_{\xi}$-regular, equivalently, the convergence lattice $L=(\mathbb{P}X,\lim_{\xi})$
is $\C_{L}$-regular.

In \cite{FredetJean}, a convergence lattice $L$ is called \emph{classical}
if it is regular with respect to the set of complemented elements
of $L$ (see \cite[Lemma 2.29]{FredetJean}). In \cite{myn.ptfreeAP}
and \cite{convframes}, $**$-regular convergence frames are considered.
They are those convergence frames that are $\Z$-regular for the set
$\Z=\{\ell\in L:\ell=\ell^{**}\}$, where $\ell^{*}$ denotes the
pseudocomplement of $\ell$ (note that every element of a frame has
a pseudocomplement). 
\begin{lem}
Let $(L,\lim_{L})$ be a convergence lattice. Then $\pt L$ is a $\Z$-regular
convergence space for $\Z=\left\{ A\subset\pt L:\exists\ell\in L\;\left(A=\ell^{\bullet}\right)\right\} $.
Moreover $\pt'L$ is $q(\Z)$-regular where $q(\Z)=\{q(A):A\in\Z\}$.
\end{lem}

\begin{proof}
$\Z$-regularity of $\pt L$ follows from $\lim_{\pt L}\F=(\lim_{L}\F^{\circ})^{\bullet}$
and $\F^{\circ}=((\F\cap\Z)^{\uparrow})^{\circ}$. Now,

\begin{eqnarray*}
\lm_{\pt'L}\F & = & \bigcup_{q[\G]\leq\F}q(\lm_{\pt L}\G)\\
 & = & \bigcup_{q[\G]\leq\F}q(\lm_{\pt L}(\G\cap\Z)^{\uparrow})
\end{eqnarray*}
 and $q(\G\cap\Z)\subset\F\cap\Z\subset\F$ so that 
\[
\lm_{\pt'L}\F=\bigcup_{q[\G\cap\Z]\subset\F\cap\Z}q(\lm_{\pt L}(\G\cap\Z)^{\uparrow})\subset\lm_{\pt'L}(\F\cap\Z)^{\uparrow}.
\]
\end{proof}
\begin{lem}
\label{lem:simplifiedpt'} If $(L,\lim_{L})$ is $\Z$-regular where
$\ell^{\bullet}$ is saturated for every $\ell\in\Z$, then 
\[
\lm_{\pt'L}\U=q\left(\lm_{\pt L}q^{-1}[\U]\right)=q(\lm_{\pt L}\W),
\]
for every ultrafilter $\U$ on $\pt'L$ on every $\W\in\mathbb{U}(q^{-1}[\U])$.
\end{lem}

\begin{proof}
Let $L$ be $\Z$-regular. Because $\U$ is an ultrafilter, $\lim_{\pt'L}\U=\bigcup_{\W\in\mathbb{U}(q^{-1}[\U])}q((\lim_{L}\W^{\circ})^{\bullet})$.
Since $\lim_{L}\W^{\circ}=\lim_{L}\uparrow(\W^{\circ}\cap\Z)$ by
$\Z$-regularity, it is enough to show that $\W^{\circ}\cap\Z=\Z\cap(q^{-1}[\U])^{\circ}$
whenever $\W\in\mathbb{U}(q^{-1}[\U])$. This follows from the observation
that $\ell\in\W^{\circ}\cap\Z$ if and only if $\ell\in\Z$ and $\ell^{\bullet}\in\W$,
and $\ell^{\bullet}$ is saturated and thus $\ell^{\bullet}\in q^{-1}[q[\W]]=q^{-1}[\U]$.
\end{proof}
In view of Lemma \ref{lem:saturatedoepnclosed}, Lemma \ref{lem:simplifiedpt'}
applies in particular if $\Z=\O_{L}$ or $\Z=\C_{L}$.
\begin{prop}
\label{prop:Oregulartopt} If $(L,\lim_{L})$ is $\O_{L}$-regular
and has enough open elements, then $\pt L$ is topological. If $\pt L$
is topological, so is $\pt'L$.
\end{prop}

\begin{proof}
Since $L$ has enough open sets,
\begin{eqnarray*}
\lm_{\pt L}(\F\cap\O_{\pt L})^{\uparrow} & = & \lm_{\pt L}(\F\cap\{u^{\bullet}:u\in\O_{L}\})^{\uparrow}\\
 & = & \left(\lm_{L}((\F\cap\{u^{\bullet}:u\in\O_{L}\})^{\uparrow})^{\circ}\right)^{\bullet}\\
 & = & \left(\lm_{L}\uparrow(\O_{L}\cap\F^{\circ})\right)^{\bullet}\\
 & = & \left(\lm_{L}\F^{\circ}\right)^{\bullet}=\lm_{\pt L}\F,
\end{eqnarray*}
because $L$ is $\O_{L}$-regular. Hence $\pt L$ is $\O$-regular,
that is, topological. As $q$ is a continuous, onto, and open map
by Corollary \ref{cor:qopenclosed}, $\pt'L$ is also topological
by Proposition \ref{prop:openimageoftop}.
\end{proof}
\begin{cor}
\label{prop:pt'sober} If $L$ is $\O_{L}$-regular, has enough open
sets, and has closed limits then $\pt'L$ is a sober topological space.
\end{cor}

\begin{proof}
Let $\U$ be an irreducible ultrafilter on $\pt'L$, that is, $\lim_{\pt'L}\U\in\U$.
In view of Lemma \ref{lem:simplifiedpt'}, $\lm_{\pt'L}\U=q(\lm_{\pt L}\W)$
for every $\W\in\mathbb{U}(q^{-1}[\U])$, so that
\[
q^{-1}(\lm_{\pt'L}\U)=q^{-1}(q(\lm_{\pt L}\W))=\lm_{\pt L}\W\in q^{-1}[\U]\subset\W
\]
 because each $\lim_{\pt L}\W$ is closed (because $L$, hence also
$\pt L$, has closed limits by Proposition \ref{prop:limLpointclosed})
hence saturated by Lemma \ref{lem:saturatedoepnclosed}. In other
words, each $\W$ is an irreducible ultrafilter so that $\lim_{\pt L}\W=\lim_{\pt L}\{x\}^{\uparrow}=(\lim_{L}x)^{\bullet}$
for some $x\in\pt L$ because, in view of Proposition \ref{lem:saturatedoepnclosed},
$\pt L$ is weakly quasi-sober. Thus $\lim_{\pt'L}\U=q((\lim_{L}x)^{\bullet})=\lm_{\pt'L}\{\lm_{L}x\}^{\uparrow}$
by Corollary \ref{cor:pt'aas}. Hence $\pt'L$ is weakly quasi-sober
and aas by Corollary \ref{cor:pt'aas}, hence weakly sober. In view
of Proposition \ref{prop:wsoberandqwsober} it is sober because it
is topological. 
\end{proof}
In view of Lemma \ref{lem:closedelmttosubset} and Proposition \ref{prop:UonXbullet-1},
$(\mathbb{P}X,\lim_{\xi})$ always has enough open elements and enough
closed elements. Moreover, it has closed limits when $\xi$ has closed
limits, in particular if $\xi$ is topological. Hence, Proposition
\ref{prop:Oregulartopt} and Corollary \ref{prop:pt'sober} apply
to the effect that:
\begin{cor}
\label{cor:pt'Ltopsober} If $(X,\xi)$ is a topological space and
$L=(\mathbb{P}X,\lim_{\xi})$ then $\pt L$ is topological and weakly
quasi-sober and $\pt'L$ is a topological sober space. 
\end{cor}

In fact, in this case $\pt'L$ is exactly the \emph{sobrification}
of the topological space $X$, as we will see in Section \ref{sec:Sobrification}
below.

\section{Sobrification and the space $\protect\pt(\mathbb{P}X,\lim_{\xi})$\label{sec:Sobrification}}

Recall that a closed subset $C$ of a convergence space is \emph{irreducible
}if $C\subset D$ or $C\subset F$ whenever $C\subset D\cup F$ where
$D$ and $F$ are closed. Let us call $C$ \emph{c-irreducible }if
it satisfies this property even if $C$ is not necessarily closed.
Note that $C$ is c-irreducible if and only if for every open sets
$O$ and $U$ intersecting $C$, $O\cap U\cap C\neq\emptyset.$ Obviously,
$C$ is c-irreducible if and only if $\cl C$ is an irreducible closed
set.
\begin{lem}
\label{lem:cirreduciblesets} In a convergence space, limits of irreducible
filters are c-irreducible and the closure of every c-irreducible set
is the limit of a $\T\xi$-irreducible ultrafilter. 
\end{lem}

\begin{proof}
If $\lim_{\xi}\F\in\F$ then $\lim_{\xi}\F$ is a c-irreducible set.
Indeed, if $\lim_{\xi}\F\subset D\cup F$ then $D\cup F\in\F\subset\F^{\#}$
so that either $D\in\F^{\#}$ or $F\in\F^{\#}$. Hence $\lim_{\xi}\F\subset D$
or $\lim_{\xi}\F\subset F$ because $D$ and $F$ are closed.

If $C$ is a c-irreducible set then $\{O\in\O_{\xi}:O\cap C\neq\emptyset\}$
is a filter base because $C$ is c-irreducible. Let $\U$ be an ultrafilter
containing this filter base. Note that $\U$ is finer than all $\N_{\xi}(x)$
for $x\in C$, so that $C\subset\lim_{\T\xi}\U$ hence $\cl C\subset\lim_{\T\xi}\U$.
Note that $\U$ is $\T\xi$-irreducible, for if $\lim_{\T\xi}\U\notin\U$
then $X\setminus C\in\U$ which is not the case. Moreover, if $t\notin\cl C$,
there is $O=X\setminus\cl C\in\N(t)$ that is disjoint from $C$ hence
does not belong to $\U$. In other words, $\lm_{\T\xi}\U=\cl C$.
\end{proof}
\begin{rem}
\label{rem:notaas} In particular, \emph{in a topological space every
irreducible closed set $C$ is the limit of an irreducible ultrafilter}--in
fact, of many, for every ultrafilter containing the filter base $\{O\in\O_{\xi}:O\cap C\neq\emptyset\}$
satisfies this condition. As a result, if $X$ is a $T_{0}$ topological
space then (\ref{eq:ptPaas}) is equivalent to every irreducible ultrafilters
being principal, a property characterized for topological spaces in
\cite[Theorem 2.2]{hoffmann1977irreducible} and more generally for
(some) convergence spaces in Corollary \ref{cor:soberTd} below to
be equivalent with $X$ being both sober and $T_{D}$. To see that
(\ref{eq:ptPaas}) implies that irreducible ultrafilters are principal,
note that if there is a free irreducible ultrafilter on a topological
space, then $C=\lim\U$ is an infinite irreducible closed set and
the non-maximal filter-base $\{O\in\O_{\xi}:O\cap C\neq\emptyset\}$
admits one, hence many finer free ultrafilters all of which admit
$C$ as their limit, so that (\ref{eq:ptPaas}) fails. Conversely,
if all irreducible ultrafilters are principal, then every irreducible
closed set has a generic point, which is necessarily unique in a $T_{0}$
topological space, so that (\ref{eq:ptPaas}) is satisfied.
\end{rem}

\begin{cor}
If $\F$ is an irreducible filter on a convergence space $(X,\xi)$
then 
\[
\cl_{\xi}(\lm_{\xi}\F)=\lm_{\T\xi}\F.
\]
\end{cor}

\begin{proof}
We only need to show that $\lim_{\T\xi}\F\subset\cl_{\xi}(\lim_{\xi}\F)$
for the reverse inclusion is always true. In view of Lemma \ref{lem:cirreduciblesets},
$\lim_{\xi}\F$ is c-irreducible and thus, for every ultrafilter $\U$
of the filter-base $\{O\in\O_{\xi}:O\cap\lim_{\xi}\F\neq\emptyset\}$,
$\cl_{\xi}(\lim_{\xi}\F)=\lim_{\T\xi}\U$. Note that the condition
$O\cap\lim_{\xi}\F\neq\emptyset$ implies that $O\in\F$, so that
$\F$ contains this filter-base. 
\end{proof}
Recall (e.g., \cite{sobrremainder}) that the \emph{sobrification
}of a $T_{0}$ topological space $X$ is the set $^{s}X$ of irreducible
closed subsets of $X$ endowed with the topology $\{^{s}O:O\in\O_{\xi}\}$
where $^{s}O=\{C\in\,^{s}X:C\cap O\neq\emptyset\}$(\footnote{to see that this defines indeed a topology, note that $^{S}(\bigcup_{O\in\B}O)=\bigcup_{O\in\B}\,^{S}O$
and that $^{S}(O\cap U)=^{S}O\cap\,^{S}U$ because an irreducible
closed sets intersects $O\cap U$ whenever it intersects open sets
$O$ and $U$.}). Because $X$ is $T_{0}$ the map $e:X\to\,^{s}X$ defined by $e(x)=\cl\{x\}$
is one-to-one. It is a dense embedding because $^{s}O\cap e(X)=e(O)$.

Note that in view of Lemma \ref{lem:cirreduciblesets} we can identify
elements of $\pt'L$ (equivalence classes of irreducible ultrafilters
with the same limit) and elements of $^{s}X$. In fact, since the
topology of $^{s}X$ is the topology induced on $^{s}X$ by the upper
topology of the space of closed subsets of $X$ ordered by inclusion
because
\[
^{s}O=\{C\in\,^{s}X:C\cap O\neq\emptyset\}=\,^{s}X\setminus\downarrow(X\setminus O),
\]
we conclude via Theorem \ref{thm:uppertopology} that $^{s}X$ is
homeomorphic to $\pt'L$ endowed with the quotient topology induced
by $q=\lim'_{L}:\pt L\to\pt'L$, but this is the the structure of
$\pt'L$, since $\pt'L$ is topological by Corollary \ref{cor:pt'Ltopsober}.
\begin{thm}
Let $(X,\xi)$ be a topological and let $L=(\mathbb{P}X,\lim_{\xi})$.
Then $\pt'L$ is the sobrification $^{s}X$ of $X$, and $e=q\circ h$
is a dense embedding in the diagram below whenever $(X,\xi)$ is $T_{0}$:
\[
\xymatrix{X\ar[r]^{h:x\mapsto\{x\}^{\uparrow}}\ar[d]_{e:x\mapsto\cl\{x\}} & \pt(\mathbb{P}X,\lim_{\xi})\ar[d]^{q}\\
^{s}X\ar@{=}[r] & \pt'(\mathbb{P}X,\lim_{\xi})
}
\]
\end{thm}

\section{The axiom $T_{D}$ in convergence spaces}

Recall that a topological space $X$ is $T_{D}$ if for every $x\in X$
there is $U\in\O(x)$ with $U\setminus\{x\}$ open (e.g., \cite{MR2868166}).
\begin{lem}
\label{lem:TD} In a topological space $X$, the following are equivalent:
\begin{enumerate}
\item $X$ is $T_{D}$;
\item for every $x\in X$ there is $U\in\O(x)$ with 
\[
\lim\{x\}^{\uparrow}\cap U=\{x\};
\]
\item for every $x\in X$ and every filter $\F$ on $X$ that converges
to $x$, there is $A\in\F$ with 
\[
\lim\{x\}^{\uparrow}\cap A\subset\{x\};
\]
\item for every $x\in X$ and every ultrafilter $\U$ on $X$ that converges
to $x$, there is $A\in\U$ with 
\[
\lim\{x\}^{\uparrow}\cap A\subset\{x\}.
\]
\end{enumerate}
\end{lem}

\begin{proof}
$(1)\then(2)$: If $X$ is $T_{D}$ and $x\in X$, pick $U\in\O(x)$
with $U\setminus\{x\}$ open. Then $\lim\{x\}\cap U=\{x\}$ because
$t\in\lim\{x\}^{\uparrow}\cap(U\setminus\{x\})$ would imply $t\in\{x\}^{\uparrow}$
which is impossible.

$(2)\then(3)$: If $x\in\lim\F$, then $\F\geq\O(x)$ and we can take
the $U$ from (1) as $A$. $(3)\then(4)$ is obvious.

$(4)\then(1):$ For every $\U\in\mathbb{U}(\N(x))$, there is $A_{\U}\in\U$
with $\lim\{x\}\cap A_{\U}\subset A_{\U}$. Hence there is a finite
subset $F$ of $\mathbb{U}(\N(x))$ with $\bigcup_{\U\in F}A_{\U}\in\N(x)$.
Let $U\in\O(x)$ with $U\subset\bigcup_{\U\in F}A_{\U}$. Then $U\setminus\{x\}$
is open. Indeed, if $t\in\lim\U\cap(U\setminus\{x\})$ for some ultrafilter
$\U$ then $U\in\U$ because $U$ is open. If $U\setminus\{x\}\notin\U$
then $U^{c}\cup\{x\}\in\U$, hence $\U=\{x\}^{\uparrow}$ and (4)
ensures a contradiction.
\end{proof}
Note that $(3)$ and $(4)$ make sense for general convergences and
are equivalent (\footnote{\begin{proof}
If $\xi$ satisfies $(4)$ and $x\in\lim\F$ then $x\in\lim\U$ for
every $\U\in\mathbb{U}(\F)$ so that there is $A_{\U}\in\U$ with
$\lim\{x\}^{\uparrow}\cap A_{\U}\subset\{x\}$. Hence there is a finite
subset $F$ of $\mathbb{U}(\F)$ with $A=\bigcup_{\U\in F}A_{\U}\in\F$.
Then 
\[
\lim\{x\}\cap A=\bigcup_{\U\in F}\lim\{x\}\cap A_{\U}\subset\{x\}.
\]
\end{proof}
}). Thus we call a \emph{convergence space} $T_{D}$ if it satisfies
$(3)$, equivalently $(4)$. As far as I know, this is the first time
this axiom is introduced for convergence spaces.

Just as for topological space, $T_{D}$ lies between $T_{0}$ and
$T_{1}$: $T_{1}$ obviously implies $T_{D}$ as in this case $\lim\{x\}^{\uparrow}=\{x\}$
for every $x$. On the other hand, if $X$ is $T_{D}$ and $\left\{ \F\in\mathbb{F}X:x\in\lim\F\right\} =\left\{ \F\in\mathbb{F}X:t\in\lim\F\right\} $
then $\{x,t\}\subset\lim\{x\}^{\uparrow}\cap\lim\{t\}^{\uparrow}$
so by $T_{D}$ there is $A\in\{x\}^{\uparrow}$ with $\lim\{t\}^{\uparrow}\cap A\subset\{t\}$
which ensures $x=t$ because $x\in\lim\{t\}^{\uparrow}\cap A$.

Note that in all the examples of Section \ref{sec:sober}, we have
$x\neq y$ with $\{x,y\}\in\lim\{x\}^{\uparrow}\cap\lim\{y\}^{\uparrow}$
so that $\{y\}^{\uparrow}$ is an ultrafilter converging to $x$ and
$y\in\lim\{x\}^{\uparrow}$. Hence $y\in\lim\{x\}^{\uparrow}\cap A$
for every $A\in\{y\}^{\uparrow}$ and the space is \emph{not} $T_{D}$.

Preservation properties for $T_{D}$ easily extend from topologies
to convergences with this definition: 
\begin{prop}
\label{prop:TDpreserve} If $\xi\leq\tau$ and $\xi$ is $T_{D}$,
so is $\tau$; a subspace of a $T_{D}$ convergence space is $T_{D}$;
a finite product of $T_{D}$ convergence spaces is $T_{D}$.
\end{prop}

\begin{proof}
If $x\in\lim_{\tau}\F\subset\lim_{\xi}\F$ there is $A\in\F$ with
$\lim_{\xi}\{x\}^{\uparrow}\cap A\subset\{x\}$ and the conclusion
follows from $\lim_{\tau}\{x\}^{\uparrow}\subset\lim_{\xi}\{x\}^{\uparrow}$.
If $A\subset X$ with inclusion map $i:A\to X$ and $x\in A\cap\lim_{\xi_{A}}\F$,
$x\in\lim_{\xi}i[\F]$ so there is $F\in i[\F]$ with $\lim_{\xi}\{x\}^{\uparrow}\cap F\subset\{x\}$
so that $\lim_{\xi_{A}}\{x\}^{\uparrow}\cap(A\cap F)\subset\{x\}$
and $(A,\xi_{A})$ is $T_{D}$. Suppose $(X_{1},\xi_{1}),\ldots(X_{n},\xi_{n})$
are $T_{D}$ convergence spaces and let $(x_{1},\ldots x_{n})\in\lim_{\Pi_{i=1}^{n}\xi_{i}}\F$,
that is $x_{i}\in\lim_{\xi_{i}}p_{i}[\F]$ for every $i$, where $p_{i}:\Pi_{i=1}^{n}X_{i}\to X_{i}$
is the projection. There is $A_{i}\in p_{i}[\F]$ with $\lim_{\xi_{i}}\{x_{i}\}^{\uparrow}\cap A_{i}\subset\{x_{i}\}$
so that $\Pi_{i=1}^{n}A_{i}\in\Pi_{i=1}^{n}p_{i}[\F]\subset\F$ and
then $\lim_{\Pi_{i=1}^{n}\xi_{i}}\{(x_{1},\ldots x_{n})\}^{\uparrow}\cap\Pi_{i=1}^{n}A_{i}\subset\{(x_{1},\ldots x_{n})\}$. 
\end{proof}
On the other hand, a countable product of topological $T_{D}$ spaces
(e.g., countably many copies of the Sierpi\'nski space) may fail
to be $T_{D}$ (See e.g., \cite{brummer69}).

However:
\begin{lem}
\label{lem:finiteTD} 
\begin{enumerate}
\item A $T_{D}$ convergence space is antisymmetric.
\item Moreover, if the space is finite the converse is true, that is, if
$\to$ is antisymmetric then the convergence is $T_{D}$.
\item If $\lim\{x\}^{\uparrow}$ is closed for every $x$ then $\to$ is
transitive. In particular if the convergence is $S_{0}$ then $\to$
is transitive, and conversely if the space is finite and of finite
depth.
\end{enumerate}
\end{lem}

Hence when the convergence is $T_{D}$ and limits of principal ultrafilters
are closed, we have a \emph{specialization order} defined by $x\geq y$
if $y\in\lim\{x\}^{\uparrow}$.
\begin{proof}
1. Assume there are points $x\neq y$ with $\{x,y\}\subset\lim\{x\}^{\uparrow}\cap\lim\{y\}^{\uparrow}$.
Then $y\in\lim\{x\}^{\uparrow}\cap A$ for every $A\in\{y\}^{\uparrow}$
and the convergence is not $T_{D}$. 

2. Conversely for a finite space, if the convergence is not $T_{D}$
there is a (principal because the space is finite) ultrafilter $\{y\}^{\uparrow}$
converging to $x$ such that $\lim\{x\}^{\uparrow}\cap A\cap(X\setminus\{x\})\neq\emptyset$
for every $A\in\{y\}^{\uparrow}$, in particular for $A=\{y\}$. Hence
$y\neq x$ and $y\in\lim\{x\}^{\uparrow}$ and thus $\{x,y\}\subset\lim\{x\}^{\uparrow}\cap\lim\{y\}^{\uparrow}$.

3. That $\to$ is transitive if $\lim\{x\}^{\uparrow}$ is closed
for every $x$ is clear. For the converse for a finite space, assume
$\F=\{F\}^{\uparrow}$ converges to $x$ and $\{x\}^{\uparrow}$ converges
to $y$. For every $t\in F$, $t\to x$ and $x\to y$, so $t\to y$
by transitivity. By finite depth, $y\in\bigcap_{t\in F}\lim\{t\}^{\uparrow}=\lim\bigcap_{t\in F}\{t\}^{\uparrow}=\lim\{F\}^{\uparrow}$.
Hence, the space is $S_{0}$ and thus limits of principal ultrafilters
are closed.
\end{proof}
Note that in view of Lemma \ref{lem:finiteTD}, a \emph{finite}, finitely
deep, weakly sober and non-sober convergence like Example \ref{exa:weaklysobernotsober}
cannot be $T_{D}$. This is in fact general:
\begin{prop}
\label{prop:TDweaklysober} In a $T_{D}$ weakly sober space every
irreducible ultrafilter is principal. If moreover the convergence
is of finite depth, then the only irreducible filters are principal
ultrafilters and thus the convergence is also sober.
\end{prop}

\begin{proof}
If $\xi$ is weakly sober and $\U$ is an irreducible filter, then
there is a (unique) $x$ with $\lim\U=\lim\{x\}^{\uparrow}\in\U$.
As $\xi$ is $T_{D}$ there is $U\in\U$ with $\lim\{x\}^{\uparrow}\cap U\subset\{x\}$.
As a result, $\{x\}\in\U$ and $\U$ is principal. Hence if $\F$
is irreducible, every ultrafilter finer than $\F$ is also irreducible,
hence principal. As a result, $\F$ is the principal filter of a finite
set $F$. If the convergence is moreover of finite depth, 
\[
\lim\F=\lim\{F\}^{\uparrow}=\bigcap_{x\in F}\lim\{x\}^{\uparrow}\in\F=\{F\}^{\uparrow}.
\]
By Lemma \ref{lem:finiteTD} (1), $F$ is a singleton. The space is
then sober.
\end{proof}
In view of Proposition \ref{prop:TDpreserve}, if $\T\xi$ is $T_{D}$
so is the finer convergence $\xi$, but the converse is false:
\begin{example}[A finitely deep $T_{D}$ convergence whose topological modification
is not $T_{D}$]
 Consider on $X=\{x,y,z\}$ the finitely deep convergence determined
by $\lim\{x\}^{\uparrow}=\{x,y\}$, $\lim\{y\}^{\uparrow}=\{y,z\}$
and $\lim\{z\}^{\uparrow}=\{z,x\}$.

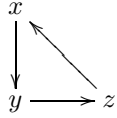
\[
\xymatrix{x\ar[d]\\
y\ar[r] & z\ar[ul]
}
\]

This convergence is $T_{D}$ By Lemma \ref{lem:finiteTD}, but its
topological modification is antidiscrete, hence not even $T_{0}$.
\end{example}

\section{Problem \ref{prob:main}}
\begin{thm}
\label{thm:main} Let $(X,\xi)$ be a convergence space.
\begin{enumerate}
\item If $\xi$ is weakly sober and $T_{D}$ then every irreducible filter
is principal.
\item If $\xi$ is $S_{0}$ and $T_{0}$ the converse is true, that is,
$\xi$ is weakly sober and $T_{D}$ whenever every irreducible ultrafilter
is principal. 
\end{enumerate}
\end{thm}

\begin{proof}
(1) is Proposition \ref{prop:TDweaklysober}.

(2) if $\xi$ is $T_{0}$ and $S_{0}$ and every irreducible ultrafilter
is principal then $\xi$ is weakly sober: if $\lim\U\in\U$ there
is $x$ with $\U=\{x\}^{\uparrow}$ and thus $\lim\U=\lim\{x\}^{\uparrow}$.
If $t\neq x$ satisfies $\lim\{t\}^{\uparrow}=\lim\{x\}^{\uparrow}$,
by $T_{0}$ there is $\F\in\mathbb{F}X$ with $\card(\lim\F\cap\{t,x\})=1$,
which is not compatible with $S_{0}$. Moreover, $\xi$ is $T_{D}$.
Suppose to the contrary that there is $x\in X$ and $\U\in\mathbb{U}X$
with $x\in\lim\U$ such that $\lim\{x\}^{\uparrow}\cap U\cap(X\setminus\{x\})\neq\emptyset$
for every $U\in\U$. In particular, $\lim\{x\}^{\uparrow}\in\U^{\#}=\U$
and $\lim\{x\}^{\uparrow}\subset\lim\U$ because $\xi$ is $S_{0}$.
Hence $\lim\U\in\U$ and $\U$ is irreducible, hence principal, that
is, $\U=\{t\}^{\uparrow}$ and $\lim\{x\}^{\uparrow}\subset\lim\{t\}^{\uparrow}$.
Hence $x\in\lim\{t\}^{\uparrow}$ and $t\in\lim\{x\}^{\uparrow}$
by taking $U=\{t\}$ in $\lim\{x\}^{\uparrow}\cap U\cap(X\setminus\{x\})\neq\emptyset$,
the argument above for the uniqueness of a generic point applies to
the effect that $x=t$. But that is not compatible with $\lim\{x\}^{\uparrow}\cap U\cap(X\setminus\{x\})\neq\emptyset$
for $U=\{t\}=\{x\}$ and we have a contradiction.
\end{proof}
\begin{cor}
\label{cor:soberTd} Let $(X,\xi)$ be a $S_{0}$ and $T_{0}$ convergence
space of finite depth. The following are equivalent:
\begin{enumerate}
\item every irreducible ultrafilter is principal;
\item every irreducible filter is principal;
\item $\xi$ is weakly sober and $T_{D}$;
\item $\xi$ is sober and $T_{D}$;
\item Every subspace of $\xi$ is sober;
\item If $\theta\geq\xi$ then $\theta$ is sober;
\item $\pt_{\mathbf{Lat}}(\mathbb{P}X,\lim_{\xi})$ is homeomorphic to $(X,\xi)$.
\end{enumerate}
\end{cor}

\begin{proof}
Equivalence between points 1 through 4 follows directly from Theorem
\ref{thm:main}, Proposition \ref{prop:TDweaklysober}.

$(1)\iff(5)$: By Proposition \ref{prop:TDweaklysober}, sober and
weakly sober are equivalent in a $T_{D}$ convergence of finite depth,
hence in all its subspaces. Let $A\subset X$ with inclusion map $i:A\to X$,
and let $\U$ be an ultrafilter on $A$ with $\lim_{\xi_{A}}\U\in\U$.
Hence $i(\lim_{\xi_{A}}\U)\in i[\U]$ and $i(\lim_{\xi_{A}}\U)\subset\lim_{\xi}i[\U]$
so that the ultrafilter $i[\U]$ on $(X,\xi)$ is irreducible. By
(1), there is $x\in X$ with $i[\U]=\{x\}^{\uparrow}$. As $A\in i[\U]$,
$x\in A$ and thus $\lim_{\xi_{A}}\U=\lim_{\xi_{A}}\{x\}^{\uparrow}$
and $A$ is weakly sober. 

Conversely, if there is a free ultrafilter $\U$ that is irreducible,
if $\lim\U$ does not have a generic point, then the space is not
weakly sober. If $\lim\U$ has a generic point, it is unique by Lemma
\ref{lem:equallimpoints}, say, $\lim\U=\lim\{x_{0}\}^{\uparrow}$.
Let $A=\lim\U\setminus\{x_{0}\}$ with the induced convergence. Then
$A\in\U$ (because $\lim\U\in\U$ and $X\setminus\{x_{0}\}\in\U$
as $\U$ is free) and $\lim_{\xi_{A}}\U=A\in\U$ but $\lim_{\xi_{A}}\U$
has not generic point. Hence, $(A,\xi_{A})$ is a non weakly sober
subspace.

$(1)\iff(6):$ If $\U\in\mathbb{U}X$ with $\lim_{\theta}\U\in\U$
then $\lim_{\xi}\U\in\U$ because $\lim_{\theta}\U\subset\lim_{\xi}\U$,
so that $\U$ is $\xi$-irreducible, hence principal. Hence it has
a unique generic point, because $\xi$ is $S_{0}$ and $T_{0}$, hence
almost antisymmetric. Thus $\theta$ is sober.

Conversely, suppose there is a free $\xi$-irreducible ultrafilter
$\U$. Let $\theta$ be defined by $x\in\lim_{\theta}\F$ if and only
if $\F=\{x\}^{\uparrow}$ whenever $x\notin\lim_{\xi}\U$ and $x\in\lim_{\theta}\F$
if and only if $\F=\U$ or $\F=\{x\}^{\uparrow}$ whenever $x\in\lim_{\xi}\U$.
Then $\theta\geq\xi$ and $\theta$ is not sober: $\U$ is $\theta$-irreducible
because $\lim_{\theta}\U=\lim_{\xi}\U\in\U$ but does not have a generic
point because singletons are closed and $\lim_{\theta}\U$ is infinite. 

$(1)\iff(7)$: Let $L=(\mathbb{P}X,\lim_{\xi})$. By $(1)$ the map
the map $h:\pt_{\mathbf{Lat}}L\to(X,\xi)$ defined by $h(\{x\}^{\uparrow})=x$
is onto and thus is an homeomorphism by Proposition \ref{prop:subspaceofptL}.

Conversely, if $\pt L$ is homeomorphic to $(X,\xi)$ but there are
free $\xi$-irreducible ultrafilters then, in view of Proposition
\ref{prop:subspaceofptL}, $\pt L$ contains an homeomorphic copy
of itself as a dense \emph{proper} subset via an homeomorphism $j:\pt L\to j(\pt L)\subsetneq\pt L$,
which is not possible, for $j(\pt L)$ is closed as an homeomorphic
image of a closed subset of $\pt L$, hence $\pt L\setminus j(\pt L)$
is open, making it impossible for $j(\pt L)$ to be dense in $\pt L$.
\end{proof}
Note that in view of Remark \ref{rem:notaas}, when $(X,\xi)$ is
topological, we can add (\ref{eq:ptPaas}) for every pair $\U,\W$
of irreducible ultrafilters as an additional equivalent condition
in Corollary \ref{cor:soberTd}.

\bibliographystyle{plain}

\begin{thebibliography}{10}
	
	\bibitem{brummer69}
	G.~Br{\"u}mmer.
	\newblock Initial quasi-uniformities.
	\newblock {\em Indag. Math.}, 31:403--409, 1969.
	
	\bibitem{D.comp}
	S.~Dolecki.
	\newblock Convergence-theoretic characterizations of compactness.
	\newblock {\em Topology and its Applications}, 125:393--417, 2002.
	
	\bibitem{DGL}
	S.~Dolecki, G.~H. Greco, and A.~Lechicki.
	\newblock Compactoid and compact filters.
	\newblock {\em Pacific J. Math.}, {\bf 117}:69--98, 1985.
	
	\bibitem{DM.book}
	S.~Dolecki and F.~Mynard.
	\newblock {\em Convergence {F}oundations of {T}opology}.
	\newblock World Scientific, 2016.
	
	\bibitem{FredetJean}
	J.~Goubault-Larrecq and F.~Mynard.
	\newblock {C}onvergence without points.
	\newblock {\em Houston Journal of Mathematics}, 46(1):227--282, 2020.
	
	\bibitem{MR3086734}
	Jean Goubault-Larrecq.
	\newblock {\em Non-{H}ausdorff topology and domain theory}, volume~22 of {\em
		New Mathematical Monographs}.
	\newblock Cambridge University Press, Cambridge, 2013.
	\newblock [On the cover: Selected topics in point-set topology].
	
	\bibitem{hoffmann1977irreducible}
	R.E. Hoffmann.
	\newblock Irreducible filters and sober spaces.
	\newblock {\em Manuscripta Math}, 22(4):365--380, 1977.
	
	\bibitem{sobrremainder}
	R.E. Hoffmann.
	\newblock On the sobrification remainder $^sx\setminus x$.
	\newblock {\em Pacific J. Math.}, 83(1):145--156, 1979.
	
	\bibitem{convframes}
	F.~Mynard.
	\newblock More pointfree convergence: on convergence frames.
	\newblock {\em in preparation}.
	
	\bibitem{myn.applofcompact}
	F.~Mynard.
	\newblock Products of compact filters and applications to classical product
	theorems.
	\newblock {\em Topology and its Applications}, 154(4):953--968, 2007.
	
	\bibitem{myn.relations}
	F.~Mynard.
	\newblock Relations that preserve compact filters.
	\newblock {\em Applied Gen. Top.}, 8(2):171--185, 2007.
	
	\bibitem{myn.ptfreeAP}
	F.~Mynard.
	\newblock Approach theory and pointfree convergence.
	\newblock {\em Top. Proc.}, 61(31-47), 2023.
	
	\bibitem{MR2868166}
	J.~Picado and A.~Pultr.
	\newblock {\em Frames and locales: Topology without points}.
	\newblock Frontiers in Mathematics. Birkh\"auser/Springer Basel AG, Basel,
	2012.
	
	\bibitem{Schroder74}
	M.~Schroder.
	\newblock Compactness theorems.
	\newblock In {\em Categorical Topology}, pages 566--577. Springer-Verlag, 1974.
	
\end{thebibliography}

\end{document}